\documentclass[11pt,a4paper]{article}          

\usepackage{amsfonts}
\usepackage{amsmath}

\usepackage{amsthm}
\usepackage{amssymb} 
\usepackage{needspace}
\usepackage{hyperref}
\usepackage{fullpage}
\usepackage{rotating}

\usepackage{tikz}
\usetikzlibrary{decorations.pathreplacing,shapes}
\usepackage{tkz-berge}
\usepackage{tkz-graph}
\usepackage{tkz-arith}
\usepackage{tkz-euclide}
\usepackage{tkz-base}
\usepackage{pgf}
\usepackage{esvect}

\usepackage{mathptmx}

\newtheorem{theorem}{Theorem}[section] 
\numberwithin{theorem}{section}
\newtheorem{lem}[theorem]{Lemma}
\newtheorem{prop}[theorem]{Proposition} 
\newtheorem{cor}[theorem]{Corollary} 
\newtheorem{rem}[theorem]{Remark} 

\numberwithin{equation}{section}

\newcommand{\V}{{\mathbb V}}
\newcommand\Prob[1]{{\mathbb{P}\left\{#1\right\}}}
\newcommand{\cU}{\mathcal{U}}
\newcommand{\VV}{\mathcal{V}}
\newcommand{\cV}{\mathcal{V}}
\newcommand{\ind}{{\bf 1}}
\newcommand{\eqd}{\,{\buildrel {\mathrm{def}} \over =}\,}
\newcommand{\Unif}{\mathop{\mathrm{Unif}}}
\newcommand\kk{\varkappa}
\newcommand\rr{\varkappa}
\def\U{\Upsilon}
\newcommand{\E}[1]{\mathbb{E} \left[#1\right]}
\newcommand{\Va}[1]{{\mathrm{Var}}\left(#1\right)}
\newcommand{\GG}{\mathcal{G}}
\newcommand{\vG}{H} 
\newcommand{\vP}{P} 
\newcommand{\vS}{S} 

\newcommand{\pr}{{\mathbb P}}

\newcommand{\numbelow}{A}

\newcommand{\BB}{\mathcal{B}}

\newcommand{\Fc}{\color{red}}

\newcommand{\m}[1]{}

\newcommand{\dirdiam}{
  \begin{tikzpicture}[baseline=-0.3ex,scale=0.25]
    \tikzstyle{vertex}=[circle,fill=black, minimum size=1pt,inner sep=1pt]
    \node[vertex] (v3) at (-0.7, 0.25){};
    \node[vertex] (v2) at (0.7,0.25){};
    \node[vertex] (v1) at (0,1.1){};
    \node[vertex] (v4) at (0,-0.6){};
    \draw[->] (v1)--(v2);
    \draw[->] (v1)--(v3);
    \draw[->] (v3)--(v4);
    \draw[->] (v2)--(v4);
  \end{tikzpicture}}

\title{Embedding small digraphs and permutations\\ in binary trees and split trees}

\date{\today}

\author{Michael~Albert \and Cecilia~Holmgren \and Tony~Johansson \and Fiona~Skerman\\
}

\begin{document}

\maketitle

\begin{abstract}
  We investigate the number of permutations that occur in random labellings of trees. This is a generalisation of the number of subpermutations occurring in a random permutation. It also generalises some recent results on the number of inversions in randomly labelled trees~\cite{cai2019inversions}. We consider complete binary trees as well as random split trees a large class of random trees of logarithmic height introduced by Devroye~\cite{MR1634354}. Split trees consist of nodes (bags) which can contain balls and are generated by a random trickle down process of balls through the nodes.

  For complete binary trees we show that asymptotically the cumulants of the number of occurrences of a fixed permutation in the random node labelling have explicit formulas. Our other main theorem is to show that for a random split tree, with probability tending to one as the number of balls increases, the cumulants of the number of occurrences are asymptotically an explicit parameter of the split tree. For the proof of the second theorem we show some results on the number of embeddings of digraphs into split trees which may be of independent interest.

\end{abstract}

\section{Introduction and statement of results}
Our two main results are the distribution of the number of appearances of a fixed permutation in random labellings of complete binary tree and split trees. Theorem~\ref{thm.cumulant} gives the distribution of the number of appearances of a fixed permutation in a random labelling of a complete binary tree. A split tree, see Section \ref{subsec.split}, is a random tree consisting of a random number and arrangement of nodes and non-negative number of balls within each node. We say an event $\mathcal{E}_n$ occurs {\em with high probability (whp)} if $\pr(\mathcal{E}_n)\to 1$ as $n\to\infty$. Theorem~\ref{thm.cumulantSPLIT} shows that for a random split tree with high probability, a result similar to Theorem \ref{thm.cumulant} holds for the number of appearances of a fixed permutation in a random labelling of the balls of the tree. We write a complete introduction and statement of results in terms of complete binary trees first before defining split trees and stating our results for split trees. This paper extends the conference paper~\cite{albert2018permutations}.

\subsection{Patterns in labelled trees}
Let $V$ denote the node set of a tree $T_n$ with $n$ nodes. Define a partial ordering on the nodes of the tree by saying that $a < b$ if $a$ is an ancestor of $b$. Suppose we have a labelling of the nodes $\pi : V \rightarrow [n]$.

We say that nodes $a$ and $b$ form an \emph{inversion} if $a<b$ and $\pi(a)>\pi(b)$. The enumeration of labelled trees with a fixed number of inversions has been studied by Gessel et al. \cite{gessel1995enumeration}, Mallows and Riordan \cite{mallows1968inversion} and Yan \cite{yan2001generalized}.

One can also extend the notion of inversions in labelled trees to longer permutations. For example, the number inverted triples in a tree $T$ with labelling $\pi$ is the number of triples of vertices $u_1<u_2<u_3$ with labels such that $\pi(u_1)>\pi(u_2)>\pi(u_3)$. In general, we say a permutation $\alpha$ appears on the $|\alpha|$-tuple of vertices $u_1, \ldots, u_{|\alpha|}$ , if $u_1<\ldots < u_{|\alpha|}$ and the induced order $\pi(u)=(\pi(u_1), \ldots, \pi(u_{|\alpha|}))$ is~$\alpha$. Write $\pi(u)\approx \alpha$ to indicate the induced order is the same: for example $527\approx213$. Permutations in labelled trees have been studied before: Anders et al. \cite{anders2019rooted} and Chauve et al. \cite{chauve2001enumerating} enumerated labelled trees avoiding permutations in the labels.

We shall be interested in the number of permutations in random labellings of trees. From now on, for fixed trees we let $\pi: V \rightarrow [n]$ be a node labelling chosen uniformly from the $n!$ possible labellings (for split trees $\pi$ is a uniformly random ball labelling). The (random) number of inversions in random node labellings of fixed trees as well as some random models of trees were studied in \cite{flajolet1998analysis,panholzer2012limiting} and extended in a recent paper \cite{cai2019inversions}. The nice paper \cite{lackner2015runs} by Lackner and Panholzer studied runs in labelled trees; i.e.\ the permutations $12\ldots k$ and $k\ldots 21$ for constant $k$. Their paper gives both enumeration results as well as a central limit law for runs in randomly labelled random rooted trees. This new paper finds approximate extensions to some of the results in \cite{cai2019inversions}.  

We now define the notation we will use. The number of inverted triples in a fixed tree $T$ is the random variable $R(321,T)=\sum_{u_1<u_2<u_3} \ind [\pi(u_1)>\pi(u_2)>\pi(u_3)]$ where the sum runs over all triples of nodes in $T$ such that $u_1$ is an ancestor of $u_2$ and $u_2$ an ancestor of $u_3$.  For a tree $T$ and uniformly random node labelling define
\[
R(\alpha, T) \eqd \sum_{u_1<\ldots<u_{|\alpha|}} 1[\pi(u)\approx\alpha],
\]
so in particular $R(21,T)$ counts the number of inversions in a random labelling of $T$. (For split trees we take $\pi$ to be a uniformly random ball labelling and the balls get a partial relation of ancestor induced by the nodes: see Section \ref{subsec.split} for details.)

Let \(d(v)\) denote the \emph{depth} of \(v\), i.e., the distance from \(v\) to the root \(\rho\). For any $u_1 < \ldots < u_{|\alpha|}$ we have $\mathbb{P}[\pi(u) \approx\alpha]  = 1/|\alpha|!$ and so it immediately follows that,

\begin{equation}\label{eq:halftpl}
  \E{R(\alpha, T)} = \sum_{u_1 < \ldots < u_{|\alpha|}} \pr\big[ \pi(u)\approx\alpha  \big] = \frac{1}{|\alpha|!} \sum_{v} \binom{d(v)}{|\alpha|-1}.
\end{equation}
For length two permutations, e.g.\ inversions, $\E{R(21,T)}=\frac12\U(T)$ the tree parameter $\U(T) \eqd \sum_{v} d(v)$ is called the {\em total path length} of $T$. We will state our results in terms of a tree parameter $\U^k_r(T)$ which generalises the notion of total path length.

Defining \(\U_{r}^k(T)\) will allows us to generalize~\eqref{eq:halftpl} to higher
moments of \(R(\alpha,T)\).
For \(r\) nodes $v_1, \ldots, v_r$ 
let $c(v_1, \ldots,v_r)$ be the
number of ancestors that they share and so \[c(v_1,\dots,v_r) \eqd \left|\left\{u \in V: u \leq v_{1}, v_{2},\ldots, v_{r}
  \right\}    \right|\] which is also the depth of the least common ancestor plus one. That is $c(v_1,\ldots,v_r)=d(v_1\lor \ldots \lor v_r)+1$ where we write $v_1\lor v_2$ for the least common ancestor of $v_1$ and $v_2$. The `off by one error' is because the root is in the set of common ancestors for any subsets of nodes but we use the convention that the root has depth $0$. Also define
\begin{equation}\label{Uk}
  \U_{r}^k(T) \eqd \sum_{v_1,\dots,v_r} c(v_{1},\dots,v_{r})\prod_{i=1}^r \binom{d(v_i)}{k-2},
\end{equation}
where the sum is over all ordered $r$-tuples of nodes in the tree and with the convention $\binom{x}{0}=1$. For a single node $v$, 
$d(v)=c(v)-1$, since \(v\) itself is counted in \(c(v)\). So $\U(T)=\U_1^2(T)-|V|$; i.e., we recover the
usual notion of total path length. The $k=2$ case recovers the $r$-total common ancestors $\U_{r}^2(T)=\sum_{v_1,\ldots, v_r}c(v_1,\ldots,v_r)$ defined in~\cite{cai2019inversions}.

Indeed the distribution of the number of inversions in a fixed tree has already been studied in~\cite{cai2019inversions}. Similarly to the way one can describe a distrubtion by giving all finite moments, we may also describe a distribution via its cumulant moments. 
The cumulants, which we by denote $\rr_r=\kk_r(X)$, are the coefficients in the Taylor expansion of the log of the moment generating function of $X$ about the origin (provided they exist)
\[ \log \mathbb{E}(e^{\xi X}) = \sum_r \kk_r \xi^r/r! \]
thus $\rr_1(X)=\E X$ and $\rr_2(X)=\Va X$. For more information on cumulants see for example \cite[Section 6.1]{JLR}.

\begin{theorem}[Cai et al.\ \cite{cai2019inversions}]\label{thm.inversions} 
  Let $T$ be a fixed tree, and denote by \(\rr_{r}=\rr_r(R(21,T))\) the \(r\)-th cumulant of \(R(21,T)\). Then for $r\geq 2$,
  \begin{align*}
    \kk_{r} = \frac{B_r(-1)^r}{r} \Big(\U_{r}^{2}(T)-|V|\Big)
  \end{align*}
  where $B_r$ denotes the $r$-th Bernoulli number.
\end{theorem}

\begin{rem}In essence Theorem~\ref{thm.inversions} (Cai et al.~\cite{cai2019inversions}) shows the $r$-th cumulant of the number of inversions is a constant times~$\U_{r}^2(T)$. Our main result on complete binary trees, Theorem~\ref{thm.cumulant} (respectively Theorem~\ref{thm.cumulantSPLIT} on split trees), shows that for any fixed permutation $\alpha$ of length $k$ for complete binary trees (and whp for split trees) the $r$-th cumulant is a constant times $\U_{r}^k(T_n)$ asymptotically. The exact constant is defined in Equation~\eqref{eq.constantD} and is a little more involved than for inversions but observe it is a function only of the moment $r$ and the length of $k=|\alpha|$ together with the first element $\alpha_1$ of the permutation $\alpha=\alpha_1\ldots\alpha_k$. 
\end{rem}

\subsection{Complete Binary trees}

We move onto stating our results. For the case of $T$ a complete binary tree on $n$ vertices we asymptotically recover Theorem~\ref{thm.inversions} (\cite{cai2019inversions}) for large $n$. Moreover we extend it to cover any fixed permutation $\alpha$ for complete binary trees.

The first of our theorems gives the distribution of the number of $\alpha$ in a random labelling of the nodes in a complete binary tree. This result formed Theorem~2 in the extended abstract version of the paper however there was an error in the definition of the constant $D_{\alpha,r}$ for $r>2$ which has now been corrected.

\needspace{10\baselineskip}
\begin{theorem}
  \label{thm.cumulant}
  Let $T_n$ be the complete binary tree with $n$ nodes and fix a permutation $\alpha=\alpha_1\ldots\alpha_k$ of length $k$. Let \(\rr_{r}=\rr_r(R(\alpha,T_n))\) be the \(r\)-th cumulant of \(R(\alpha,T_n)\). Then for $r\geq 2$, there exists a constant $D_{\alpha,r}$ depending only on $\alpha$ and $r$ such that,
  \begin{equation*}
    \kk_{r} = D_{\alpha,r}\U_{r}^{k}(T_n)+o\big(\U_{r}^{k}(T_n)\big).
    \label{kkIT}
  \end{equation*}
\end{theorem}
{
An explicit formula for $D_{\alpha,r}$ is derived in Equation~\eqref{eq.constantD} and in the Appendix on page~\pageref{table} we list values of $D_{\alpha,r}$ for permuatations $\alpha$ of length at most 6 and moments $r\in\{1,\ldots,5\}$. The explicit formula~\eqref{eq.constantD} implies the following corollary. 
}

\begin{cor}\label{cor.treecount} Let $T_n$ be the complete binary tree with $n$ nodes. For permutations $\alpha$ of length 3, the variance is
  \[
  \V(R(\alpha,T_n)) = \left\{\begin{array}{ll}
      \frac{1}{45}\U^3_2(T_n)
      (1+o(1)) & \mbox{for $\alpha=123,132,312,321$}\\ 
      \frac{1}{180}\U^3_2(T_n)
      (1+o(1)) & \mbox{for $\alpha=213,231$}\\
    \end{array}\right.
  \]
  and more generally for $\alpha=\alpha_1\alpha_2\ldots\alpha_k$,
  \[
  \V(R(\alpha,T_n))=\begin{cases} 
    \frac{1}{((k-1)!)^2}\Big(\frac{1}{2k-1}-\frac{1}{k^2}\Big)\U_2^{k}(1+o(1))& \mbox{for $\alpha_1\in \{1,k\}$}\\
    \Big(\frac{1}{(2k-1)(k-\alpha_1)!(k+\alpha_1-2)!}- \frac{1}{(k!)^2}\Big)\U_2^{k}(1+o(1))& \mbox{for $\alpha_1\in \{2,\ldots,k-1\}$}.
  \end{cases}
  \]
\end{cor}

\begin{rem} The methods in the proofs are very different for inversions and general permutations. In~\cite{cai2019inversions}, the method takes advantage of a nice independence property of inversions\m{was independence property of permutations}. For a node $u$ let $I_u$ be the number of inversions involving $u$ as the top node: $I_u=|\{w: u<w, \pi(u)>\pi(w)\}|$. Then the $\{I_u\}_u$ are independent random variables and $I_u$ is distributed as the uniform distribution on~$\{0,\ldots, |T_u|\}$ where $T_u$ is the subtree rooted at $u$, see Lemma~1.1 of~\cite{cai2019inversions}.

  Without a similar independence property for general permutations our route instead uses nice properties on the number of embeddings of small digraphs in both complete binary trees and, whp, in split trees. This property allows us to calculate the $r$-th moment of $R(\alpha, T)$ directly from a sum of products of indicator variables as most terms in the sum are zero or negligible by the embedding property.
\end{rem}

\subsection{Split trees}\label{subsec.split}

Split trees were first defined in \cite{MR1634354} and were introduced to encompass many families of trees that are frequently used in algorithm analysis, e.g., binary search trees \cite{MR0142216}, $m$-ary search trees
\cite{MR0216622} and quad trees \cite{Finkel1974}. The full definition is given below but note that a split tree is a random tree which consists of nodes (bags) each of which contains a number of balls. We will study the number of occurences of a fixed subpermutation $\alpha$ in a random ball labelling of the split tree.

The random split tree $T_n$ has parameters $b, s, s_0, s_1, \VV$ and $n$. The integers $b, s, s_0, s_1$ are required to satisfy the inequalities
\begin{equation}
  2 \le b, \quad 0 < s, \quad 0\leq s_0\leq s, \quad 0\leq bs_1\leq s+1-s_0.
  \label{eq:split:para}
\end{equation}
and $\VV=(V_1,\dots,V_b)$ is a random non-negative vector
with \(\sum_{i=1}^b V_i = 1\)  (the components $V_i$ are probabilities).

We define $T_n$ algorithmically. 
Consider the infinite $b$-ary tree $\cU$, and view each node as a
bucket or bag with capacity $s$. Each node (bag) $u$ is assigned an independent copy 
$\cV_u$ of the random split vector $\cV$.
Let
$C(u)$ denote the number of balls in node (bag) $u$, initially setting $C(u) = 0$ for all $u$. Say that
$u$ is a {\em leaf} if $C(u) > 0$ and $C(v) = 0$ for all children $v$ of
$u$, and {\em internal} if 
$C(v) > 0$ for some proper descendant $v$, i.e., \(v > u\). We add $n$ balls labeled $\{1,\dots,n\}$ to $\cU$
one by one. The $j$-th ball is added by the following ``trickle-down'' procedure.
\begin{enumerate}
\item Add $j$ to the root.
\item While $j$ is at an internal node (bag) $u$, choose child $i$ with probability $V_{u,i}$, where
  $\cV_u=(V_{u,1}, \dots,V_{u,b})$ is the split vector at $u$, and move $j$ to child $i$.
\item If $j$ is at a leaf $u$ with $C(u) < s$, then $j$ stays at $u$ and we set $C(u) \leftarrow C(u) + 1$.

  If $j$ is at a leaf with $C(u) = s$, then the balls at $u$ are distributed among $u$ and its
  children as follows. We select $s_0\leq s$ of the balls uniformly at random to stay at $u$. Among
  the remaining $s+1-s_0$ balls, we uniformly at random distribute $s_1$ balls to each of the $b$
  children of $u$. Each of the remaining $s+1-s_0-bs_1$ balls is placed at a child node 
  chosen independently at random according to the split vector assigned to \(u\).
  This splitting process is repeated for any child which receives more than $s$ balls.
\end{enumerate}

Once all $n$ balls have been placed in $\cU$, we obtain $T_n$ by deleting
all nodes $u$ such that the subtree rooted at $u$ contains no balls. Note
that 
an internal node (bag) of $T_n$ contains exactly $s_0$ balls, while a leaf
contains a random amount in $\{1,\dots,s\}$. We can assume that the components $V_i$ of the split vector \(\cV\)  are identically distributed. If this was not the case they can anyway be made identically distributed by using a random permutation, see \cite{MR1634354}. Let $V$ be a random variable with this distribution. We assume, as previous authors,
that $\Prob{\exists i : V_i = 1} < 1$.
For this paper we will also require that the internal node (bag) capacity $s_0$ is at least one so that there are some internal balls to receive labels.

For example, if we let \(b=2, s=s_0=1, s_{1}=0\) and \(\cV\) have the
distribution of \( (U, 1-U)\) where \(U \sim \Unif[0,1]\), then
we get the well-known binary search tree.

An alternate definition of the random split tree is as follows.  Consider an infinite \(b\)-ary tree \(\cU\).
The split tree \(T_n\) is constructed by distributing \(n\) balls (pieces of information) among
nodes of \(\cU\). For a node \(u\), let \(n_u\) be the number of balls stored in the subtree rooted
at \(u\). Once \(n_u\) are all decided, we take \(T_n\) to be the largest subtree of \(\cU\) such
that \(n_u > 0\) for all \(u \in T_n\). Let \(\VV_u = (V_{u,1},\dots,V_{u,b})\) be the independent copy of \(\VV\) assigned to \(u\). Let
\(u_1,\dots,u_b\) be the child nodes of \(u\). Conditioning on \(n_u\) and
\(\VV_u\),
if $n_u\le s$, then $n_{u_i}=0$ for all $i$;
if $n_u>s$, then
\[
(n_{u_1}, \dots, n_{u_{b}}) 
\sim 
\mathrm{Mult}(n-s_0-bs_{1}, V_{u,1}, \dots, V_{u, b}) 
+
(s_{1}, s_{1},\dots, s_{1}),
\]
where \(\mathrm{Mult}\) denotes multinomial distribution, and \(b, s, s_{0}, s_{1}\) are integers
satisfying \eqref{eq:split:para}. Note that we have \(\sum_{i=1}^b n_{u_i} \le n\) (hence the ``splitting'').
Naturally for the root \(\rho\), \(n_{\rho}=n\). Thus the distribution of \((n_{u}, \VV_{u})_{u \in V(\cU)}\) is completely defined. 

The balls inherit a partial order from the partial ordering of the nodes in the split tree. We write $u_1<u_2$ if node $u_1$ is an ancestor of node $u_2$, $u_1>u_2$ if $u_2$ is an ancestor of $u_1$ and finally $u_1 \perp u_2$ is neither $u_1$ nor $u_2$ is an ancestor of the other node. 
For balls $j_1, j_2$ in nodes (bags) $u_1, u_2$ respectively $j_1<j_2$ if $u_1<u_2$ and $j_1 \perp j_2$ if $u_1 \perp u_2$. We say that balls $j_1, j_2$ are incomparable, $j_1\perp j_2$ if they are in the same node (bag).

This next theorem is our other main result. We determine the distribution of the number of occurences of a fixed subpermutation in a random ball labelling of the split tree. Denote the random variable for the number of occurences of $\alpha$ in a uniformly random ball labelling of split tree $T_n$ by $R(\alpha, T_n)$.
\needspace{3\baselineskip}
\begin{theorem}
  \label{thm.cumulantSPLIT}
  Fix a permutation $\alpha=\alpha_1\ldots\alpha_k$ of length $k$. Let $T_n$ be a split tree with split vector $\VV=(V_1,\ldots,V_b)$ and $n$ balls. Let \(\rr_{r}=\rr_r(R(\alpha,T_n))\) be the \(r\)-th cumulant of \(R(\alpha,T_n)\). For $r\geq 2$ the constant $D_{\alpha,r}$ is defined in Equation~\eqref{eq.constantD}. Whp the split tree $T_n$ has the following property. 
  \begin{equation*}
    \kk_{r} = D_{\alpha,r}\U_{r}^{k}(T_n)+o\big(\U_{r}^{k}(T_n)\big).
  \end{equation*}
\end{theorem}
Our theorem says the following. Generate a random split tree $T_n$, whp it has the property that the random number of occurrences of any fixed subpermutation in a random ball labelling of $T_n$ has variance and higher cumulant moments approximately a constant times a `simple' tree parameter of $T_n$.

\begin{rem}
  We may contrast this with Theorem~1.12 of~\cite{cai2019inversions}. That theorem states the distribution of the number of inversions in a random split tree; where the distribution is expressed as the solution of a system of fixed point equations. Determining the distribution of $\U_{r}^{k}(T_n)$ would extend Theorem~1.12 of~\cite{cai2019inversions} about inversions to general permutations.
\end{rem}

\subsection{Embeddings of small digraphs}

Certain classes of digraphs, defined below, will be important in the proof of Theorem~\ref{thm.cumulant}. Loosely the digraphs we will consider are those that may be obtained by taking $r$ copies of the directed path $\vP_k$ and iteratively fusing pairs of vertices together. It will also matter how many  embeddings each digraph has into the complete binary tree. In Proposition~\ref{prop.onlystars} we show the counts for most digraphs in such a class are of smaller order than the counts of a particular set of digraphs in the class. The main work in the proof of this proposition is to show that the number of embeddings of any digraph $\vG$, up to a constant factor, depends only on the numbers of two types of vertices in~$\vG$. We separate this result out as a theorem, Theorem~\ref{lem.lcacountsBin}, which we prove in Section \ref{embeddingbinary}. 

We now define the particular notion of embedding small digraphs into a tree which will be important. Define a digraph to be a simple graph together with a direction on each edge. We shall consider only acyclic digraphs i.e.\ those without a directed cycle. 

In the complete binary tree we have a natural partial order, the ancestor relation, where the root is the ancestor of all other nodes. Any fixed acyclic digraph also induces a partial order on its vertices where $v<u$ if there is a directed path from $v$ to $u$. For an acyclic digraph $\vG$, define $[\vG]_{T_n}$ to be the number of embeddings $\iota$ of $\vG$ to distinct nodes in $T_n$ such that the partial order of vertices in $\vG$ is respected by the embedding to nodes in $T_n$ under the ancestor relation. 
\[
[\vG]_{T_n} \eqd |\{\iota : V(\vG)\rightarrow V(T_n) \mbox{ injective such that if $u<v$ in $\vG$ then $\iota(u)<\iota(v)$ in $T_n$}\}|.
\]
Observe that the inverse of embedding $\iota^{-1}$ need not respect relations. If $u \perp v$ in $\vG$, i.e.\ $u,v$ are incomparable in $\vG$ then we can embed so that $\iota(u)<\iota(v)$, $\iota(u)>\iota(v)$ or $\iota(u)\perp \iota(v)$ in $T_n$. For an example of this take the digraph $\dirdiam$ and denote by $P_\ell$ the rooted path on $\ell$ nodes. Notice that in $\dirdiam$ two of the vertices are incomparable but the vertices of the digraph can be embedded into the nodes of a path which are completely ordered. The counts are $[\dirdiam]_{P_4}=2$ and in general $[\dirdiam]_{P_\ell}=2\binom{\ell}{4}$.

A particular star-like digraph $\vS_{k,r}$ will be important. This is the digraph obtained by taking $r$ directed paths of length $k$ and fusing their source vertices into a single vertex. Alternatively the theorem can be stated in terms of star counts as $[\vS_{|\alpha|,r}]_{T_n}= \U_{r}^{|\alpha|}(T_n)(1+o(1))$: see Lemma \ref{lem.upsilon}.

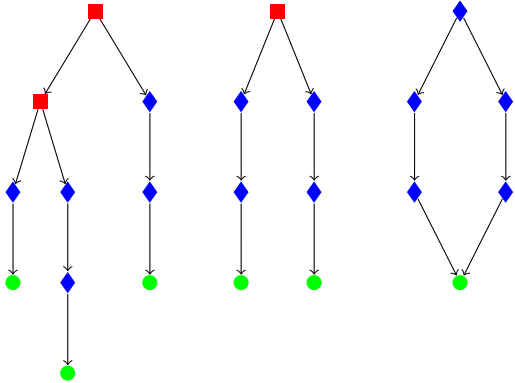
\begin{figure}[h]
  \centering
  \begin{tikzpicture}[baseline=1ex,scale=1.2]

    \tikzstyle{vertex}=[circle,fill=black, minimum size=2pt,inner sep=2pt]
    \tikzstyle{vertexCA}=[rectangle,fill=red, minimum size=2pt, draw=none, inner sep=2.7pt]
    \tikzstyle{vertexA}=[diamond,aspect=0.7,fill=blue, draw=none, inner sep=1.7pt]

    \node[vertexCA, fill=red] (r1) at (6, 4){};
    \node[vertexA, fill=blue] (r2) at (5.6, 3){};
    \node[vertexA, fill=blue] (r3) at (5.6, 2){};
    \node[vertex, fill=green] (r4) at (5.6, 1){};
    \node[vertexA, fill=blue] (r2r) at (6.4, 3){};
    \node[vertexA, fill=blue] (r3r) at (6.4, 2){};
    \node[vertex, fill=green] (r4r) at (6.4, 1){};

    \node[vertexCA, fill=red] (m1) at (4, 4){};
    \node[vertexA, fill=blue] (mr2) at (4.6, 3){};
    \node[vertexA, fill=blue] (mr3) at (4.6, 2){};
    \node[vertex, fill=green] (mr4) at (4.6, 1){};

    \node[vertexCA, fill=red] (ml2) at (3.4, 3){};
    \node[vertexA, fill=blue] (mlr3) at (3.7, 2){};
    \node[vertexA, fill=blue] (mlr4) at (3.7, 1){};
    \node[vertex, fill=green] (mlr5) at (3.7, 0){};

    \node[vertexA, fill=blue] (mll3) at (3.1, 2){};
    \node[vertex, fill=green] (mll4) at (3.1, 1){};

    \node[vertexA, fill=blue] (rm1) at (8, 4){};
    \node[vertexA, fill=blue] (rmr2) at (8.5, 3){};
    \node[vertexA, fill=blue] (rmr3) at (8.5, 2){};
    \node[vertex, fill=green] (rm4) at (8, 1){};

    \node[vertexA, fill=blue] (rml2) at (7.5, 3){};
    \node[vertexA, fill=blue] (rml3) at (7.5, 2){};

    \draw[->] (rm1)--(rmr2);
    \draw[->] (rmr2)--(rmr3);
    \draw[->] (rmr3)--(rm4);

    \draw[->] (rm1)--(rml2);
    \draw[->] (rml2)--(rml3);
    \draw[->] (rml3)--(rm4);

    \draw[->] (m1)--(mr2);
    \draw[->] (mr2)--(mr3);
    \draw[->] (mr3)--(mr4);

    \draw[->] (m1)--(ml2);
    \draw[->] (ml2)--(mll3);
    \draw[->] (mll3)--(mll4);
    \draw[->] (ml2)--(mlr3);
    \draw[->] (mlr3)--(mlr4);
    \draw[->] (mlr4)--(mlr5);

    \draw[->] (r1)--(r2);
    \draw[->] (r2)--(r3);
    \draw[->] (r3)--(r4);
    \draw[->] (r1)--(r2r);
    \draw[->] (r2r)--(r3r);
    \draw[->] (r3r)--(r4r);

  \end{tikzpicture}
  \caption{An example of a directed acyclic graph $\vG$ with sink (green $\color{green}\bullet$), `ancestor' (blue~{\small\color{blue} $\blacklozenge$})  and `common-ancestor' (red {\small \color{red} $\blacksquare$}) nodes indicated by colour and shape. This particular digraph is in $\GG_{4,7}$ and it appears in the seventh moment calculations of $R(\alpha,T)$ for $|\alpha|=4$. }\label{fig.sinksetc}
\end{figure}

A vertex in a directed graph is a \emph{sink} if it has zero out-degree. Define $\numbelow_0(H)\subseteq V(\vG)$ to be the set of sinks in digraph $H$. Recall that a directed acyclic graph defines a partial order on the vertices: $v<u$ if there is a directed path from $v$ to $u$. If $v<u$ we say that $u$ is a descendant of $v$. Define $\numbelow_1(\vG)\subseteq V(\vG)$ to be the vertices with exactly one descendant which is a sink. 
We will call vertices in $\numbelow_1$ \emph{ancestors} as they are ancestors of a single sink. Define $\numbelow_2(H)$ to be the remainder $\numbelow_2(\vG)=V(\vG)\backslash \{\numbelow_0\cup \numbelow_1\}$. We call those in $\numbelow_2$ \emph{common-ancestors} as they are the common ancestor of at least two sinks (see Figure~\ref{fig.sinksetc}). Observe if $\vG$ is a directed forest then the sinks are the leaves. However, $\vG$ need not be a forest and indeed a sink may have indegree more than one as in the rightmost sink in Figure~\ref{fig.sinksetc}.

\needspace{6\baselineskip}
For the split tree $T_n$ and an acyclic digraph $\vG$, define $[\vG]_{T_n}$ to be the number of embeddings $\iota$ of vertices in $\vG$ to distinct balls in $T_n$ such that the partial order of vertices in $\vG$ is respected by the embedding to balls in $T_n$ under the ancestor relation. 
\begin{theorem}\label{lem.lcacountsBin}
  Let $\vG$ be a fixed directed acyclic graph and let $T_n$ be the complete binary tree of height $m$ with $n=2^{m+1}-1$ vertices.
  Then writing $|\numbelow_0|=|\numbelow_0(\vG)|$ for the number of sink (green) vertices and $|\numbelow_1|=|\numbelow_1(\vG)|$ for the number of `ancestor' (blue) vertices
  \[
  [\vG]_{T_n}=\Theta(n^{|\numbelow_0|}(\ln n)^{|\numbelow_1|}).
  \]\end{theorem}
This improves on bounds provided in the conference version of this paper~\cite{albert2018permutations}. Similarly for split trees we show that the expected \m{added expected} number of embeddings of a fixed acyclic digraph $\vG$, to constant factors, depends only on the number of sink and `ancestor' vertices in $\vG$.  
\begin{theorem}\label{thm.lcacountsSplit}
  Let $\vG$ be a fixed directed acyclic graph and let $T_n$ be a split tree with split vector $\VV=\{V_1,\ldots,V_b\}$ and $n$ balls.
  Then writing $|\numbelow_0|=|\numbelow_0(\vG)|$ for the number of sink (green) vertices and $|\numbelow_1|=|\numbelow_1(\vG)|$ for the number of `ancestor' (blue) vertices there exist constants $c=c(\vG)$ and $c'=c'(\vG)$ such that for large enough $n$,

  \[
  \E{[\vG]_{T_n}}\leq c n^{|\numbelow_0|}(\ln n)^{|\numbelow_1|}
  \]
  and whp
  \[
  [\vG]_{T_n}\geq c' n^{|\numbelow_0|}(\ln n)^{|\numbelow_1|}.
  \]
\end{theorem}

In the extended abstract version of this paper~\cite{albert2018permutations}, in Lemma~7, we proved the weaker upper bound that for constant $c''$ whp $[\vG]_{T_n}\leq c''n^{|\numbelow_0|}(\ln n)^{|\numbelow_1|}(\ln \ln n)^{|\numbelow_2|}$, i.e.\ a dependence also on the number of `common-ancestor'  (red) vertices in $\vG$. It is a little trickier to prove the new upper bound. However, we are rewarded by a tighter bound on the number of embeddings; the expected number of embeddings is now determined only by the numbers of sink (green) and `ancestor' (blue) vertices up to constant factors. It would be interesting to obtain tail bounds on the number of embeddings of small digraphs in a random split tree and we leave this as an open question.

\needspace{8\baselineskip}
\section{Embeddings of small digraphs into the complete binary tree}\label{embeddingbinary}

In this section we prove Theorem \ref{lem.lcacountsBin} concerning upper and lower bounds on the number of embeddings of a fixed digraph $\vG$, thought of as constant, into a complete binary tree $T_n$ with $n$ vertices.

We prove the lower bound of Theorem~\ref{lem.lcacountsBin} first as the upper bound will require some preparatory lemmas.
\begin{proof}(of lower bound of Theorem~\ref{lem.lcacountsBin})
  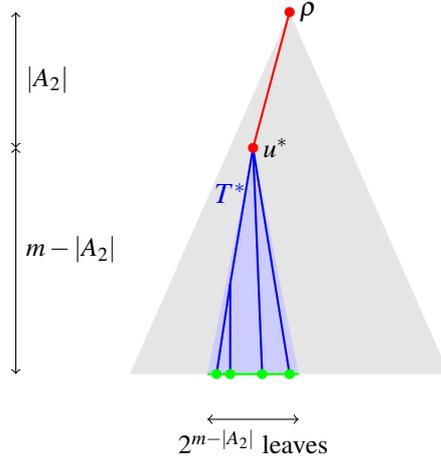
\begin{figure}
    \centering
    \begin{tikzpicture}[scale=0.6]
      \path [fill=gray, opacity=0.2]  (3.2,6) -- (2.2,1) -- (0.5,1) -- (4,9) -- (7.5,1) --(4.2,1) --(3.2,6);    
      \path [fill=blue, opacity=0.2] (2.2,1) -- (3.2,6) -- (4.2,1);

      \draw[red,thick] (3.2,6) -- (4,9);

      \draw[blue, thick] (2.4,1) -- (2.7,3);
      \draw[blue, thick] (2.7,1) -- (2.7,3);
      \draw[blue, thick] (3.2,6) -- (2.7,3);
      \draw[blue, thick] (3.4,1) -- (3.2,6);
      \draw[blue, thick] (4,1) -- (3.2,6);

      \draw[green, thick] (2.2,1) -- (4.2,1);
      \filldraw[color=green, fill=green](2.4,1) circle [radius=0.1];
      \filldraw[color=green, fill=green](2.7,1) circle [radius=0.1];
      \filldraw[color=green, fill=green](3.4,1) circle [radius=0.1];
      \filldraw[color=green, fill=green](4,1) circle [radius=0.1];
      
      \filldraw[color=red, fill=red](3.2,6) circle [radius=0.1];
      \node [right] at (3.2,6) {\small $u^*$};
      \filldraw[color=red, fill=red](4,9) circle [radius=0.1];
      \node [right] at (4,9) {\small $\rho$};

      \draw[<->] (-2,6) -- (-2,9);
      \node [right] at (-2,7.5) {\small $|\numbelow_2|$};
      \draw[<->] (-2,1) -- (-2,6);
      \node [right] at (-2,3.75) {\small $m-|\numbelow_2|$};

      \draw[<->] (2.2,0) -- (4.2,0);
      \node [below] at (3.2,0) {\small $2^{m-|\numbelow_2|}$ leaves};
      
      \node [blue] at (2.7,5) {$T^*$};

    \end{tikzpicture}
    \caption{Schematic for the lower bound construction in Theorem~\ref{lem.lcacountsBin}. The colours indicate the positions in the complete binary tree to which the `common-ancestor' (red), `ancestor' (blue) and sink (green) vertices are embedded. Recall $\numbelow_2=\numbelow_2(\vG)$ denotes the set of `common-ancestor' vertices of $\vG$.}\label{fig.sketchconstruct}
  \end{figure}We restrict attention to embeddings where all `common-ancestors' of $\vG$ are embedded very near the root of $T_n$, the sink vertices are embedded to leaves of $T_n$ and the `ancestor' vertices are placed on the path between the root of $T_n$ and the leaf to which their descendant sink was embedded (see Figure~\ref{fig.sketchconstruct}). There are sufficiently many such embeddings to obtain the lower bound. In fact we restrict a little further to make it easy to check all the embeddings are valid.

  The first task is to embed the vertices in $\numbelow_2$ close to the root in such a way that $\numbelow_2$ is embedded to ancestors of the nodes to which $\numbelow_1$ and $\numbelow_0$ are embedded and also such that the ordering within the vertices in $\numbelow_2$ is preserved. As $\vG$ is an acyclic digraph the directed edges define a partial order on all vertices of $\vG$ and in particular for those in $\numbelow_2$. Thus this relation can be extended to a total order. Fix such a total order $<_*$ on $V(\vG)$, one which extends the partial order on $V(\vG)$,\m{responding to review: 'Fix some total order' - state explicitly that the total order must be an extension of hte order of vertices of H as described in the previous sentence} and relabel vertices in $\numbelow_2$ so that $v_1<_* \ldots<_* v_{|\numbelow_2|}$. Thus we may embed $v_1$ to the root $\rho$ in $T_n$ and each $v_{i+1}$ to a child of the node to which $v_i$ was embedded and the relation between vertices in $\vG$ will be preserved by their embedding in $T_n$; i.e. we may embed $\numbelow_2$ to the nodes on the path from $\rho$ to some $u^*$ at depth $|\numbelow_2|-1$. Fix such a node~$u^*$ and let $T^*$ be the subtree of $T_n$ from $u^*$.

  Label the sinks $\numbelow_0=\{s_1, \ldots, s_{|\numbelow_0|}\}$ and vertices in $\numbelow_1$ according to which sink they are the ancestors of $\numbelow_1^i\eqd \{v\in \numbelow_1 \; : \; v< s_i\}$.

  We obtain a subcount of $[\vG]_{T_n}$ by embedding $\numbelow_2$ onto the path from $\rho$ to $u^*$, embedding $\numbelow_0$ to leaves of $T^*$ and then for each $i$ in turn embedding vertices in $\numbelow_1^i$ on the path from $u^*$ to the embedding of $s_i$. There are $m-|\numbelow_2|-1$ vertices on the path from $s_i$ to $u^*$ and at most $|A_1|$ of them already have an ancestor vertex embedded onto to them (i.e.\ from $\numbelow_1^j$ for some $j<i$). 
  Thus
  \begin{equation}\label{eq.worksbecausefixed}
  [\vG]_{T_n}\geq \binom{2^{m-|\numbelow_{2}|}}{|\numbelow_0|}\prod_{i} \binom{m-|\numbelow_{2}|-|\numbelow_1|-1}{|\numbelow_1^i|}
  \end{equation}
  where the first binomial coefficient counts the number of ways to embed $\numbelow_0$ and the $i$-th binomial coefficient in the product counts the ways to embed $\numbelow_1^i$. Now because $\vG$ is fixed $|\numbelow_0|$, $|\numbelow_1|$ and $|\numbelow_2|$ are all $O(1)$. Hence for large $m$ the RHS of Equation \eqref{eq.worksbecausefixed} has first term of order $\Theta(2^{m|\numbelow_0|})$ and the product over $i$ is of order $\Theta(m^{\sum_i |\numbelow_1^i|})=\Theta(m^{|\numbelow_0|})$ so the lower bound follows.\end{proof}

The key observation to prove the upper bound in Theorem~\ref{lem.lcacountsBin} is that for most pairs of nodes in a complete binary tree their least `common ancestor' is very near the root. We make the required condition precise in the assumption of the next lemma, and show it implies the upper bound on the number of embeddings of $\vG$. It then suffices to prove that the condition holds for complete binary trees. This allows us to recycle the lemma to show the corresponding result in split trees.\

Define $c(u_1, u_2)$ to be the number of `common ancestors' of nodes $u_1$ and $u_2$.

\begin{lem}\label{lem.commontoHcounts} Let $\vG$ be a fixed directed acyclic graph and let $T_n$ be any tree with $n$ nodes and height $m$. Then writing $|\numbelow_0|=|\numbelow_0(\vG)|$ for the number of sink (green) vertices, $|\numbelow_1|=|\numbelow_1(\vG)|$ for the number of `ancestor' (blue) vertices and $|\numbelow_2|=|\numbelow_2(\vG)|$ for the number of `common-ancestor' (red) vertices,
  \[
  [\vG]_{T_n} \leq m^{|\numbelow_1|}  n^{|\numbelow_0|-2} \sum_{u_i,u_{j}} c(u_i, u_{j})^{|\numbelow_2|}
  \]
  where the sum is over ordered pairs of distinct nodes in $T_n$. 
\end{lem}

\begin{proof}
  Label the sinks $\numbelow_0=\{s_1, \ldots, s_{|\numbelow_0|}\}$ and vertices in $\numbelow_1$ according to which sink they are the ancestors of $\numbelow_1^i\eqd \{v\in \numbelow_1 \; : \; v< s_i\}$. Similarly partition `common-ancestor' vertices into disjoint sets $\{A_2^{i,j}\}_{1\leq i <j \leq |\numbelow_0|}$ according to the lexicographically least pair of sinks $s_i$ and $s_j$ for which it is an ancestor. Formally a vertex $v\in A_2$ is in $A_2^{i,j}$ if $v$ is the ancestor of sinks $s_i$ and $s_j$ but not an ancestor of a sink $s_k$ for $k< \max\{i,j\}$. \\


  Suppose sinks $s_i$ and $s_j$ are embedded to vertices $u_i$ and $u_j$ in $T_n$. Then to complete the embedding of ancestors of $s_i$, vertices in $\numbelow_1^i$ must be embedded to ancestors of $u_i$ in $T_n$ and there are at most $d(u_i)$ options. Likewise vertices in $\numbelow_2^{i,j}$ i.e.\ `common-ancestors' of sinks $s_i$ and $s_j$ must be embedded to a common ancestor of $u_i$ and $u_j$ in the tree. Thus, recalling $c(u_i, u_j)$ denotes the number of common ancestors of $u_i$ and $u_j$,

  \begin{equation}
    \label{eq.countsBin0a}
    [\vG]_{T_n} \leq \sum_{u_1, \ldots, u_{|\numbelow_0|}} \prod_i \binom{d(u_i)}{|\numbelow_1^i|} \prod_{i\neq j} \binom{c(u_i,u_{j})}{|\numbelow_2^{i,j}|}.
  \end{equation}
  where the sum is over distinct nodes $u_1,\ldots, u_{|\numbelow_0|}$ and the product $i\neq j$ is over pairs $u_i, u_{j}$ in $u_1,\ldots, u_{|\numbelow_0|}$. Fix a particular embedding of the sinks to $u_1,\ldots,u_{|\numbelow_0|}$ and we shall bound both terms in the product in~\eqref{eq.countsBin0a}. Recall that for the (blue) `ancestor' vertices, $\sum_i|\numbelow_1^i|= |\numbelow_1|$ so $\prod_i \binom{d(u_i)}{|\numbelow_1^i|}\leq (\max_i d(u_i))^{|\numbelow_1|}$. It will suffice to use the trivial bound that all vertices have depth at most the height of the tree, i.e.\ $\max_i d(u_i) \leq m$. And so,
  \[ \prod_i \binom{d(u_i)}{|\numbelow_1^i|} \leq m^{|\numbelow_1|}. \]
  Similarly, for the (red) `common-ancestor' vertices $\sum_{i\neq j} |\numbelow_2^{i,j}|=|\numbelow_2|$ as the sets $A_2^{i,j}$ are disjoint. Thus
  \[\prod_{i\neq j} \binom{c(u_i,u_{j})}{|\numbelow_2^{i,j}|} \leq  \max_{i\neq j} c(u_i, u_{j})^{|\numbelow_2|} \leq \sum_{i\neq j} c(u_i, u_{j})^{|\numbelow_2|}.\] 
  Hence substituting the bounds above into the expression in~\eqref{eq.countsBin0a},

  \begin{equation}
    \label{eq.countsBin1}
    [\vG]_{T_n} 
    \leq m^{|\numbelow_1|}   \sum_{u_i, u_{j}} c(u_i, u_{j})^{|\numbelow_2|} \sum_{u_1, \ldots, u_{|\numbelow_0|} \backslash u_i, u_{j}} \!\!\!\!\!\!\!\! \ind
    \;\;\; \leq \;\;\; m^{|\numbelow_1|}  n^{|\numbelow_0|-2} \sum_{u_i,u_{j}} c(u_i, u_{j})^{|\numbelow_2|} 
  \end{equation}
  which is the required result.
\end{proof}

There is one more result we need and then the upper bound in Theorem~\ref{lem.lcacountsBin} will follow very fast.

\begin{lem}\label{lem.cherryBinary} Let $\vG$ be a fixed directed acyclic graph and let $T_n$ be a complete binary tree with $n$ vertices and height $m$. Then for any positive integer~$\ell$,
  \[
  \sum_{u_1,u_2} \ind[c(u_1,u_2)\geq \ell] \leq 2^{-\ell+1}n^2.
  \]
   the sum is over ordered pairs of distinct nodes in $T_n$
\end{lem}
\begin{proof} 
  Associate with each vertex $v \in V(T_n)$ a binary string of length at most $m$ in the usual way: the root has string $\varnothing$, children of the root are labelled $0$ and $1$ and two vertices in the same subtree at depth $d$ have the same initial $d$-length substring. Now $\sum_{u_1,u_2} \ind[c(u_1, u_2)\geq \ell]$ is precisely the number of ordered pairs which share a common $(\ell-1)$-length initial substring in their labels; i.e. ordered pairs with both vertices in the same depth $(\ell-1)$ subtree.

  Let $T_1^{\ell-1}, \ldots, T_{2^{\ell-1}}^{\ell-1}$ be the subtrees at depth $\ell-1$. Since $T_n$ is a complete binary tree $|T_i^{\ell-1}|=2^{m-\ell+1}-1$. Recall $n=2^{m+1}-1$ and so $|T_i^{\ell-1}| 
  \leq n2^{-\ell}.$ Now 
  \[ \sum_{u_1,u_2} \ind[c(u_1, u_2)\geq \ell]  = \sum_{i=1}^{2^{\ell-1}} |T_i^{\ell-1}|^2 \leq 
  n^2 2^{-\ell+1}  
  \]
  as required.
\end{proof}

\begin{proof}(of upper bound in Theorem~\ref{lem.lcacountsBin})
  \m{\Fc cut proof length in half}
  Observe Lemma~\ref{lem.cherryBinary} implies 
  \begin{equation}
    \notag
    \sum_{u_i, u_{j}} c(u_i, u_{j})^{|\numbelow_2|} \leq \sum_{u_i, u_{j} } \sum_{\ell=1}^\infty \ind[c(u_i, u_{j}) \geq \ell]\ell^{|\numbelow_2|} \leq n^{2}\sum_{\ell=1}^\infty (\tfrac{1}{2})^{\ell-1} \ell^{|\numbelow_2|}.
  \end{equation}
  Since $|\numbelow_2|$ is a constant the sum $\sum_{\ell=1}^\infty (\tfrac{1}{2})^{\ell}\ell^{|\numbelow_2|}$ converges to a constant, say $\beta=\beta(|\numbelow_2|)$. Thus by Lemma~\ref{lem.commontoHcounts} we get
  \[
  [\vG]_{T_n} \leq m^{|\numbelow_1|}n^{|\numbelow_0|-2}\sum_{u_i, u_j} c(u_i, u_{j})^{|\numbelow_2|} 
  \leq \beta m^{|\numbelow_1|}n^{|\numbelow_0|} =O(m^{|\numbelow_1|}n^{|\numbelow_0|}).
  \]
\end{proof}

\needspace{8\baselineskip}
\section{Embeddings of small digraphs into the split trees}\label{embeddingsplittrees}

In this section we prove Theorem \ref{thm.lcacountsSplit} concerning upper and lower bounds on the number of embeddings of a fixed digraph $\vG$, thought of as constant, into a random split tree with $n$ balls. We begin by briefly listing some results on split trees from the literature that will be useful for us.

For split vector $\VV$ define $\mu=\sum_{i}\E{V_i \ln V_i}$. The average depth of a node is $\sim \frac{1}{\mu}\ln n$~\cite[Cor 1.1]{holmgren2012novel}. Moreover almost all nodes are very close to this depth. Define a node $v$ to be \emph{good} if it has depth
\begin{equation*}\label{eq.splitgooddepth}
  |d(v)-\frac{1}{\mu}\ln n| \leq \ln^{0.6} n
\end{equation*}
and then whp $1-o(1)$ proportion of the nodes in the split tree are good \cite[Thm 1.2]{holmgren2012novel}. That whp in a split tree all good nodes have a $\Theta(\ln n)$ depth and almost all nodes are good is the only result about split trees required for the proof of the lower bound on $[\vG]_{T_n}$ in Theorem~\ref{thm.lcacountsSplit}. For the upper bound we need a bit more.

We will apply Proposition \ref{verybadnodes} below which is stated as Remark 3.4 in \cite{holmgren2012novel} (this remark refers to the proof of \cite[Thm 1.2]{holmgren2012novel} which is stated above).

\begin{prop}\label{verybadnodes} Let $T_n$ be a split tree with $n$ balls.
  For any constant $r > 0$ there is a constant $K > 0$, such that the expected number of nodes with $d(v) \geq K \ln n$ is $O(\frac{1}{n^{r}})$.
\end{prop}

We will use Proposition \ref{verybadnodes} as well as the property that most pairs of balls have their least common ancestor node very close to the root which we prove in Lemma~\ref{lem.notmanycommonsSplit}.

\begin{figure}
  \centering
  \begin{tikzpicture}[scale=0.6]
    \path [fill=blue, opacity=0.2] (0.8,2) -- (4,6) -- (7.2,2);
    
    \path [fill=green, opacity=0.2] (0,1) -- (0.8,2) -- (7.2,2) -- (8,1);
    \path [fill=green, opacity=0.2] (0,1) -- (0.8,0) -- (7.2,0) -- (8,1);
    \path [fill=gray, opacity=0.2]  (4,-4) -- (7.2,0) --(0.8,0) --(4,-4);

    \draw[red,thick] (4,6) -- (5,9);

    \draw[blue, thick] (1.4,0.2) -- (3.1,4);
    \draw[blue, thick] (2.7,1.6) -- (3.1,4);
    \draw[blue, thick] (4,6) -- (3.1,4);
    \draw[blue, thick] (4.4,0.5) -- (4,6);
    \draw[blue, thick] (6.5,1.2) -- (4,6);

    \filldraw[color=green, fill=green](1.4,0.2) circle [radius=0.1];
    \filldraw[color=green, fill=green](2.7,1.6) circle [radius=0.1];
    \filldraw[color=green, fill=green](4.4,0.5) circle [radius=0.1];
    \filldraw[color=green, fill=green](6.5,1.2) circle [radius=0.1];
    
    \filldraw[color=red, fill=red](4,6) circle [radius=0.1];
    \node [right] at (4,6) {\small $u^*\!\!,$       $\tilde{n}$ balls};
    \filldraw[color=red, fill=red](5,9) circle [radius=0.1];
    \node [right] at (5,9) {\small $\rho,$        $n$ balls};

    \draw[<->] (-2,6) -- (-2,9);
    \node [right] at (-2,7.5) {\small $|\numbelow_2|$};
    \draw[<->] (-2,1) -- (-2,6);
    \node [right] at (-2,3.75) {\small $\tfrac{1}{\mu} \ln \tilde{n}$};
    \draw[<->] (10,1) -- (10,0);
    \draw[<->] (10,1) -- (10,2);
    \node [right] at (10,0.6) {\small $\ln^{0.6} \tilde{n}$};
    \node [right] at (10,1.6) {\small $\ln^{0.6} \tilde{n}$};

    \node [blue] at (2.7,5) {$T^*$};

  \end{tikzpicture}
  \caption{Schematic for the construction in lower bound of Theorem~\ref{thm.lcacountsSplit}. The colours indicate the positions in the split tree to which the `common-ancestor' (red), `ancestor' (blue) and sink (green) vertices are embedded. Recall $\numbelow_2=\numbelow_2(\vG)$ denotes the set of `common-ancestor' vertices of $\vG$.}\label{fig.sketchconstructsplit}
\end{figure}
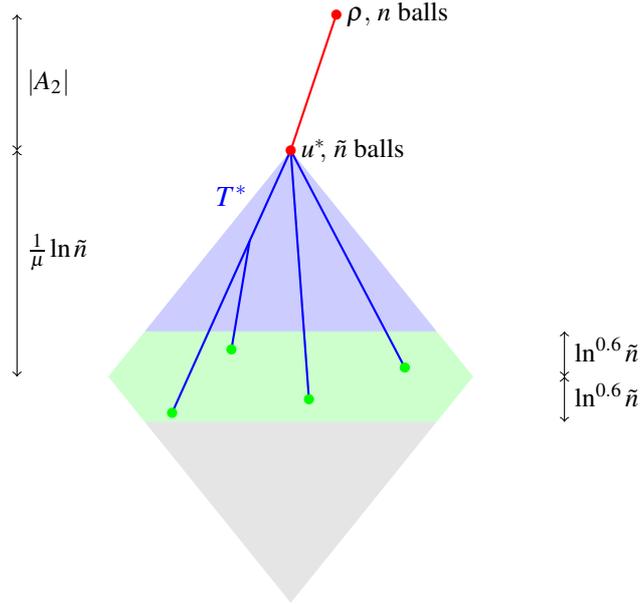

We begin with the lower bound, the upper bound is proven at the end of this section on page~\pageref{proof.lcacountsSplit}.

\begin{proof} (of the lower bound of Theorem~\ref{thm.lcacountsSplit}) 

  We describe a strategy to embed $\vG$ into $T_n$. The details of the proof are then to show that whp this strategy can be followed to obtain a valid embedding of $\vG$ and that there are sufficiently many different such embeddings to achieve the lower bound.

  The idea is as follows: first embed `common-ancestor' vertices along a path to some node $u^*$ near the root of $T_n$ so that the subtree from $u^*$ has $\tilde{n}$ balls where this $\tilde{n}$ is a constant proportion of the total number of balls $n$. Now consider the split tree with $\tilde{n}$ balls and embed `ancestor' and sink vertices into that. Embed sink vertices to `good' balls in the tree (i.e. depth very close to the expected depth) and the `ancestor' vertices to balls which are in nodes on the path between $u^*$ and the embedding of that ancestor's descendant. See Figure~\ref{fig.sketchconstructsplit}.

  We embed the `common-ancestor' vertices, $\numbelow_2(\vG)$, to the balls in the nodes on the path between a node, $u^*$ say, at depth $|\numbelow_2|-1$ and the root, using one ball per node. This is so far effectively the same as in the binary case. And  we will later embed the sink and `common-ancestor' vertices to balls in the subtree $T_{u^*}$.

  We need to confirm there is some node $u^*$ at depth $L=|\numbelow_2|-1$ with $\tilde{n}$ balls in its subtree. Each node (bag) has capacity at most $s_0$ (internal nodes) or $s$ (leaves) and there are at most $(b^{L+1}-1)$ nodes, a constant number, at depth less than $L$, so $n-O(1)$ balls remaining. These balls are shared between $b^L$, a constant, number of subtrees $T_u$. Hence by pigeon-hole principle some vertex $u^*$ has $\tilde{n}=\Theta(n)$ balls in its subtree.

  Now work in the split tree $T_{\tilde{n}}$. Embed the sink vertices to any balls in good nodes $v_1, \ldots, v_{|\numbelow_0|}$ in the split tree so these have depth $\Omega(\ln \tilde{n})$. There are $\Theta(\tilde{n}^{|\numbelow_0|})$ ways to embed them. In $\vG$ label the sink vertices $s_1, \ldots, s_{|\numbelow_0|}$ and $\numbelow_1^j\subset \numbelow_1(\vG)$ to be the `ancestor' vertices with $s_j$ as their lone descendant. Vertices in $\numbelow_1^j$ can be embedded to balls anywhere between $v_j$ and $u^*$ and so there are $\Theta((\ln \tilde{n})^{|\numbelow_1^j|})$ ways to do that for each $j$. All up there are $\Omega(\tilde{n}^{|\numbelow_0|}(\ln \tilde{n})^{|\numbelow_1|})$ ways to embed $\numbelow_0(\vG)\cup\numbelow_1(\vG)$ into balls of $T_{\tilde{n}}$. But now as $\tilde{n}=\Theta(n)$ we are done.
\end{proof}

The rest of this section is devoted to proving the upper bound of Theorem~\ref{thm.lcacountsSplit}. To prove the upper bound on the expected number of embeddings of a fixed digraph into a split tree we begin by proving the split tree analogue of Lemma~\ref{lem.commontoHcounts} which was for complete binary trees. Define $c_n(b_1, b_2)$ to be the number of node common ancestors of balls $b_1$ and $b_2$. The lemma shows that the number of embeddings of $\vG$ to balls in $T_n$ can be bounded above by a function of the number of balls, the height of the tree and the number of node common ancestors. Note that the following lemma is deterministic and is true for any instance of a split tree.

\begin{lem}\label{lem.commontoHcountsSplit} Let $\vG$ be a fixed directed acyclic graph and let $T_n$ be a split tree with $s_0>0$, $n$ balls and height $m$. Then writing $|\numbelow_0|=|\numbelow_0(\vG)|$ for the number of sink (green) vertices, $|\numbelow_1|=|\numbelow_1(\vG)|$ for the number of `ancestor' (blue) vertices and $|\numbelow_2|=|\numbelow_2(\vG)|$ for the number of `common-ancestor' (red) vertices,
  \[
  \label{eq.countsBin1Split}
  [\vG]_{T_n} \leq s_0^{|\numbelow_1|+|\numbelow_2|}m^{|\numbelow_1|}  n^{|\numbelow_0|-2} \sum_{b_i,b_{i'}} c_n(b_i, b_{i'})^{|\numbelow_2|}
  \]
  the sum is over ordered pairs of distinct balls in $T_n$
\end{lem}

\begin{proof}
  As in the proof of Lemma~\ref{lem.commontoHcounts}, label the sinks $\numbelow_0=\{s_1, \ldots, s_{|\numbelow_0|}\}$ and vertices in $\numbelow_1$ according to which sink they are the ancestors of $\numbelow_1^i\eqd \{v\in \numbelow_1 \; : \; v< s_i\}$. Also let $\numbelow_2^{ij}$ be the `common-ancestor' vertices in $\numbelow_2$ which are ancestors of both sink $s_i$ and $s_j$.

  Suppose sinks $s_i$ and $s_j$ are embedded to balls $b_i$ and $b_{i'}$ in $T_n$. Then to complete the embedding ancestors of $s_i$, i.e. vertices in $\numbelow_1^i$ must be embedded balls in node ancestors of $b_i$ in $T_n$ and there are at most $s_0d(b_i)$ options as each node ancestor of $b_i$ has $s_0$ balls. Likewise vertices in $\numbelow_2^{i,j}$ i.e.\ common-ancestors of sinks $s_i$ and $s_j$ must be embedded to balls in common ancestor nodes of $b_i$ and $b_j$ in the tree. Thus, 

  \begin{equation}
    \notag
    [\vG]_{T_n} \leq \sum_{b_1, \ldots, b_{|\numbelow_0|}} \prod_i \binom{s_0d(b_i)}{|\numbelow_1^i|} \prod_{i\neq i'} \binom{s_0c_n(b_i,b_{i'})}{|\numbelow_2^{i,i'}|}.
  \end{equation}
  where the sum is over distinct balls $b_1,\ldots, b_{|\numbelow_0|}$ and the product $i\neq i'$ is over pairs $b_i, b_{i'}$ in $b_1,\ldots, b_{|\numbelow_0|}$. The expression above is very similar to Equation \eqref{eq.countsBin1} in the proof of Lemma~\ref{lem.commontoHcounts} and the proof follows now in an identical way so we omit the details. Notice the upper bound for split trees simply picks up an additional factor of $s_0^{|\numbelow_1|+|\numbelow_2|}$. 
\end{proof}

\begin{lem}\label{lem.trickle}
  Let $j$ and $j'$ be any two distinct balls, and $v$ a node with split vector $\VV^v=(V_1^v,\ldots,V_b^v)$. Let $y$ be the probability that balls $j$ and $j'$ pass to the same child node of node $v$ conditional on the event that both balls reach node $v$. (We say a ball passes to a child node whether it stays at that child or continues further down the tree via that child node). Then,
  \[
  y\leq \sum_{i=1}^b (V_i^v)^2
  \]
\end{lem}

\begin{proof}
  If a ball $j$ reaches node $v$ there are three possible scenarios
  \begin{itemize}
  \item (i) ball $j$ is chosen as one of the $s_0$ balls to remain at node $v$ when all $n$ balls have been added to the tree.
  \item (ii) ball $j$ is chosen as one of the $bs_1$ balls which are distributed uniformly so each child of $v$ receives $s_1$ of them. 
  \item (iii) ball $j$ chooses a child of $v$ with probabilities given by the split vector $\VV^v$.
  \end{itemize}

  For each of these possible scenarios we give the probability that balls $j, j'$ pass to the same child of node $v$. Observe that swapping the scenarios for $j,j'$ gives the same probability so we list only one possibility. We summarise these in a table and then provide the proof of each line below the table.\\

  $
  \begin{array}{l|l|l|llllll}
    (i)\hspace{5mm}\;\;\;\;\;\;\; & (ii)\hspace{5mm}\;\;\;\;\;\;\; & (iii)\hspace{5mm}\;\;\;\;\;\;\; &  \mbox{Probability that $j,j'$ pass to same child} \\
    \hline
    &&&\\
    j,j'& & & 0 \\
    &&&\\
    j& j'& & 0 \\
    &&&\\
    j& & j'& 0 \\
    &&&\\
    &  j,j' & & \frac{s_1-1}{bs_1-1} \\
    &&&\\
    & j & j'& \frac{1}{b} \\
    &&&\\
    & & j,j'& \sum_i V_i^2 \\
  \end{array}
  $
  \medskip
  \vspace{4mm}

  Now, if either or both of the balls stay at node $v$ then self-evidently they cannot pass to the same child of $v$, thus the situations indicated in the first three rows have probability zero. 

  The first interesting case is if both balls are in situation (ii), i.e.\ are both chosen to be part of the $bs_1$ nodes that are distributed uniformly such that each child receives $s_1$ balls. Fix a child of $v$, the number of ways both $j, j'$ pass to that child is $\binom{s_1}{2}$; and thus there are $bs_1(s_1-1)/2$ ways for $j,j'$ to pass to the same child of $v$. Then simply divide by $bs_1(bs_1-1)/2$ to get the probability that $j, j'$ pass to the same child of $v$. This finishes this case.

  The next interesting case is if ball $j$ is in situation (ii) and ball $j'$ is in situation (iii). In this case ball $j'$ goes to each child $v$ with probability indicated by the split vector. The probability that ball $j$ goes to the same node as $j'$ is $1/b$; and indeed it didn't matter the probability with which $j'$ passes to each child of $v$.

  The last case to consider is if both $j, j'$ are in situation (iii), i.e.\ they pass to child $i$ of node $v$ with probability $V_i$ as given by the split vector. Thus the probability they both go to child $i$ of node $v$ is $\sum_i V_i^2$; and the probability they pass to the same child of $v$ is then simply the sum over the children of $v$ as required.

  After justifying each line in the table it now suffices only to show that $\frac{s_1-1}{bs_1-1} < \frac{1}{b} \leq \sum_i V_i^2$. The first is immediate, 

  \[
  \frac{s_1-1}{bs_1-1}=\frac{1}{b}-\frac{b-1}{b(bs_1-1)}
  \]
  and the second follows by Jensen's inequality.
\end{proof}

We write $c_n(j,j')$ to denote the number of nodes which are common ancestors of balls $j,j'$ and $c_n(j)$ the number of nodes which are ancestors of ball $j$, including the node containing ball $j$. Similarly, write $c_n(u)$ to be the number of nodes which are ancestors of node $u$ including node $u$ itself. Lastly denote by $j \lor_n j'$ the node which is the least common-ancestor of balls $j$ and $j'$; note if $j$ and $j'$ are in the same node then this node is $j\lor_n j'$. Observe that the number of nodes which are ancestors of a ball is one more than the depth $c_n(j)=d(j)+1$ and similarly $c_n(j,j')=d(j\lor_n j')+1$.

After recalling this notation, we can use it to express the probability $y$ in the statement of Lemma~\ref{lem.trickle}. Observe that the event that the balls $j$ and $j'$ both reach node $v$ can be expressed as $j,j'\geq v$ or equivalently $(j \lor_n j) \geq v$.

\needspace{2\baselineskip}
Now $y$ was defined as the probability that balls $j$ and $j'$ pass to the same child node of node $v$ conditional on the event that both balls reach node $v$ and conditional on node $v$ having split vector $\VV^v=(V_1^v,\ldots,V_b^v)$. So
\[
y=\pr\bigg[c_n(j, j')\geq c_n(v)+1 \;\; \big| \;\;j,j'\geq v,\;\; \VV^v \bigg].
\]
We may now also state the required lemma for split trees (this lemma plays a very similar role to the bound proven for $\sum_{u_1,u_2} \ind [ c(u_1,u_2) \geq \ell]$  in the proof of Theorem~\ref{lem.lcacountsBin} for complete binary trees).

\begin{lem}\label{lem.notmanycommonsSplit}
  Let $j, j'$ be any two distinct balls in the split tree with split vector $\VV=(V_1, \ldots, V_b)$. For $\ell \geq 1$, 
  \[
  \pr\bigg[ c_n(j,j')\geq \ell+1\bigg] \leq \mathbb{E}\big[{\sum_i V_i^2}\big]^{\ell}.
  \]
\end{lem}

\begin{proof}
  The idea is to establish, using Lemma~\ref{lem.trickle}, the probability that two balls follow the same path through the tree to some specified level given they followed the same path through the tree to the level before. We condition on $\{\mathcal{V}^v\}_v$ the set of all split vectors in the split tree. For $\ell\geq 1$

  \begin{eqnarray*}
    &&\pr \bigg[ c_n(j,j') \geq \ell+1 \;\; \big| \;\; c_n(j, j')\geq \ell, \;\; \{\mathcal{V}^v\}_v\bigg]\\
    && = \sum_{u\; : \; c_n(u)=\ell} \pr \bigg[ c_n(j,j') \geq \ell+1 \;\; \big| \;\; j, j'\geq u, \;\; \mathcal{V}^u\bigg] \\
    &&\;\;\;\;\;\;\times \pr \bigg[ j,j' \geq u \;\; \big| \;\; c_n(j, j')\geq \ell, \;\; \{\mathcal{V}^v\}_{v: c(v)<\ell}\bigg].
  \end{eqnarray*}
The first term is less than $\sum_i (V_i^u)^2$ by Lemma~\ref{lem.trickle}. For the second term note the following. If balls $j$ and $j'$ have at least $\ell$ common ancestors then their least common ancestor, the node $j\lor_n j'$ must have at least $\ell$ common ancestors. In particular $j\lor_n j'$ itself or a node on the path from $j\lor_n j'$ to the root must have precisely $\ell$ ancestors and so,
  \begin{equation}\label{eq.sumtoone}
    \sum_{u\; : \; c_n(u)=\ell} p_u \eqd \sum_{u\; : \; c_n(u)=\ell} \pr\left[ j, j' \geq u \;\; \middle| \;\; c_n(j,j')\geq \ell, \;\; \{\mathcal{V}^v\}_{v: c(v)<\ell} \right]= 1.
  \end{equation}
  (Another way to see this is that for $j$ and $j'$ to have at least $\ell$ common ancestors there must be some node $u$ which is an ancestor of both $j$ and $j'$ such that node $u$ has precisely $\ell$ ancestors.) Hence we get that

  \begin{eqnarray}
    \notag &&\pr \left[ c_n(j,j') \geq \ell+1 \;\; \middle| \;\; c_n(j, j')\geq \ell, \;\; \{\mathcal{V}^v\}_{v: c_n(v)<\ell} \right]\\
    \label{eq.puineq}&& \leq \sum_{u\; : \; c_n(u)=\ell} p_u \sum_i (V_i^v)^2.
  \end{eqnarray}
  where $\sum_u p_u=1$ and also the $p_u$ depend only on split vectors for nodes $v$ with $c_n(v)<\ell$, i.e. closer to the root than node $u$ and so the $p_u$ are independent of the $\{\mathcal{V}^w\}_{w\; : \; c_n(w)=\ell}$. We can now calculate the probability that balls $j, j'$ have $\ell+1$ ancestors conditioned on having $\ell$ by taking expectations (over split vectors) and using the tower property of expectations.

  \begin{eqnarray*}
    &&\!\!\!\!\!\!\!\!\!\!\!\!\!\!\!\!\!\!\!\!\!\!\!\!\pr\Big[c_n(j, j)\geq \ell+1 \;\; \big| \;\;j,j'\geq \ell\;\; \Big]\\
    &=& \E { \ind\{c_n(j,j') \geq \ell+1 \} \;\; \big| \;\; c_n(j, j')\geq \ell, \;\; \{\mathcal{V}^v\}_{v: c_n(v)<\ell} } 
    \\
    & \leq& \sum_{u\; : \; c_n(u)=\ell}  p_u \sum_i \E{(V_i^u)^2}\\
    & = &\mathbb{E}\Big[\sum_i V_i^2\Big].
  \end{eqnarray*}
  where the inequality in the third line followed by~\eqref{eq.puineq}. We are basically done. Notice that the root is the ancestor of any two balls, so the event $c_n(j, j')\geq 1$ has probability one and we have our `base case'. Hence 
  \begin{eqnarray*}
    &&\!\!\!\!\!\!\!\!\!\!\!\!\!\!\!\!\!\!\!\!\!\!\!\!\pr\Big[ c_n(j,j')\geq \ell+1\Big]\\
    &=&\pr\Big[c_n(j,j)\geq 1\Big] \prod_{h=1}^\ell \pr\Big[c_n(j,j')\geq h+1 \;\; \big| \;\; c_n(j, j')\geq h \; \Big]\\
    &\leq& \big(\mathbb{E}\Big[\sum_i V_i^2\Big] \big)^{\ell}
  \end{eqnarray*} as required.\end{proof}
The previous lemma implies the next proposition almost immediately. \m{\Fc have deleted most of the paragraph that was here.}

\begin{prop}\label{probabilisticsplittrees} Let $C>0$ be any constant and let $T_n$ be a split tree with $n$ balls. Then there exists a constant $\beta>0$ such that 
  \[
  \E{ \sum_{b_1, b_2} c_n(b_1, b_2)^{C} }\leq \beta n^2,
  \]
  where the sum is over balls $b_1, b_2$.
\end{prop}

\begin{proof}
  By Lemma~\ref{lem.notmanycommonsSplit}, there exists a constant $a<1$ such that for any positive integer~$\ell$,
  \[
  \sum_{b_1,b_2} \ind[c(b_1,b_2)\geq \ell] \leq a^{\ell-1}n^2.
  \]
  hence as earlier in the proof of the upper bound in Theorem~\ref{lem.lcacountsBin} this implies 
  \begin{equation}
    \notag
    \sum_{b_1, b_2} c(b_1, b_2)^{C} \leq \sum_{b_1, b_2 } \sum_{\ell=1}^\infty \ind[c(b_1, b_2 )\geq \ell]\ell^{C} \leq n^{2}\sum_{\ell=1}^\infty a^{\ell-1} \ell^{C}.
  \end{equation}
  and again since $C$ and $a<1$ are constants the sum $\sum_{\ell=1}^\infty a^{\ell}\ell^{C}$ converges to a constant, say $\beta=\beta(a, 
C)$ and we are done.\end{proof}

We are now ready to prove our upper bound on the expected number of embeddings. 

\begin{proof}\label{proof.lcacountsSplit} (of the upper bound of Theorem~\ref{thm.lcacountsSplit}) 
  Fix a digraph $\vG$, and we will show that there exists a constant $c=c(\vG)$ such that
  \begin{equation}\label{eq.toshowembed}
    \E{[\vG]_{T_n}} \leq c n^{|\numbelow_0|}(\ln n)^{|\numbelow_1|}.
  \end{equation}
  It is important to have a strong bound on the likely height of the split tree. We apply Proposition \ref{verybadnodes}. Choose $K'$ such that $\pr(h(T_n) >  K' \ln n) \leq n^{-|\vG|-1}$. Let $\BB$ denote the (bad) event that $h(T_n)> K'\ln n$, and denote by $\BB^c$ the complement of this event.

  Define random variable $X=X(T_n)$ to be $X=\sum_{b_1, b_2} c_n(b_1,b_2)^{|\numbelow_2|}$. Observe that because $X$ is non-negative and by law of total expectation $\E{X \;|\; \BB^c} \leq \E{X}/\pr(\BB^c)$ and so, by Proposition \ref{probabilisticsplittrees}, for $n$ large enough,
  \begin{equation}\label{eq.expCondOnShort}
    \E{X \;|\; \BB^c} \leq \beta n^2/(1-n^{-|\vG|-1}).
  \end{equation}\needspace{3\baselineskip}
  \noindent Now by Lemma~\ref{lem.commontoHcountsSplit}
  \[[\vG]_{T_n} \leq s_0^{|\numbelow_1|+|\numbelow_2|} h(T_n)^{|\numbelow_1|}  n^{|\numbelow_0|-2} X(T_n) 
  \]
  In particular, by conditioning on $\BB^c$: the event that the height being less than $K' \ln n$, and by Equation~\eqref{eq.expCondOnShort},
  \[\E{[\vG]_{T_n} \; | \; \BB^c} \pr(\BB^c) < s_0^{|\numbelow_1|+|\numbelow_2|} \beta (K'\ln n)^{|\numbelow_1|}  n^{|\numbelow_0|}. 
  \] It remains now to bound the expected number of embeddings conditioning on $\BB$, $\E{[\vG]_{T_n} \; | \; \BB}$. We may use a very simple bound that for any tree with $n$ balls, $\vG$ can be embedded at most $n^{|\vG|}$ times, as each vertex in $\vG$ embedded to one of the $n$ balls in the tree. This suffices as now $\E{[\vG]_{T_n} \; | \; \BB} \pr(\BB) \leq n^{-1}$. Hence we may take $c=c(\vG)$ to be $2s_0^{|\vG|}K'^{|\numbelow_1}\beta$ and we have shown the Equation \eqref{eq.toshowembed} as required.
\end{proof}

\needspace{8\baselineskip}
\section{Embeddings: stars are more frequent than other connected digraphs}
After having proved the some properties of embedding counts for our two classes of trees, complete binary trees and split trees, we show these imply the desired results on cumulants of the number of appearances of a permutation in the node labellings of complete binary trees, respectively ball labellings in split trees.

Say a sequence of trees $T_n$ with $n$ nodes (respectively balls) is \emph{explosive} if for any fixed acyclic digraph $\vG$
\[
\Omega(n^{|\numbelow_0|}(\ln n)^{|\numbelow_1|})=[\vG]_{T_n}= o(n^{|\numbelow_0|}(\ln n)^{|\numbelow_1|+1}).
\]
Thus Section~2 was devoted to showing complete binary trees are explosive and Section~3 to showing split trees are explosive whp. This section proves the cumulant results using only this explosive property of the tree classes. The first result, Proposition~\ref{prop.onlystars}, shows that the number of embeddings of most digraphs we will need to consider are of smaller order than the number of embeddings of a particular digraph the `star' $S_{k,r}$ which we define below. The other result of this section is to show the asymptotic number of embeddings of $S_{k,r}$ is asymptotically the same as our extended notion of path length $\U_r^k(T_n)$ in Lemma~\ref{lem.upsilon}.

The set $\GG_{k,r}$ is the set of acyclic digraphs which may be obtained by taking $r$ copies of the path $\vP_k$ and iteratively fusing pairs of vertices together. Likewise labelled $\vG'$ in $\GG_{k,r}'$ are those obtained by fusing together $j$ labelled paths $\vP_k$ keeping both sets of labels when a pair of vertices are fused. The set $\GG'_{3,2}$ is illustrated in Figure~\ref{fig.GG32}.

Formally let $\GG_{k, r}$ be the set of directed acyclic graphs $\vG$ on $(k-1)r$ edges (allowing parallel edges), such that the edge set can be partitioned into $r$ directed paths $\vP_1,\dots,\vP_r$, each on $k-1$ edges. For $\vG\in\GG_{k,r}$ write $\vG'$ for $\vG$ together with a labelling $V_1, \ldots, V_r$, where $V_i$ are the $k$ vertices in $\vP_i$ (note some vertices have multiple labels). Likewise write $\GG'_{k,r}$ for the labelled set of graphs. 

Denote by $\vS_{k,j}$ the digraph composed by taking $j$ copies of the path $\vP_k$ and fusing the~$j$ source vertices into a single vertex. We shall refer to this as a star graph but note it is only really stars if $k=2$.

\begin{figure}
  \centering
  \begin{tikzpicture}[baseline=1ex,scale=0.52]\hspace{-3mm}
    \tikzstyle{vertex}=[circle,fill=black,inner sep=1.4pt
    ]
    \tikzstyle{vertexCA}=[rectangle,fill=red, minimum size=2pt, draw=none, inner sep=1.89pt
    ] 
    \tikzstyle{vertexA}=[diamond,aspect=0.7,fill=blue, draw=none, inner sep=1.19pt]
    

    \node[coordinate] (row1) at (6, 8) {}; 
    \node[coordinate] (row2) at (6, 4) {}; 
    \node[coordinate] (row3) at (6, 0) {};

    \node[coordinate] (v1) at ($(row1) + (-.2, 0)$) {}; 
    \node[vertexA, fill=blue] (aa1l) at (v1) {};
    \node[vertexA, fill=blue] (aa1r) at ($(v1) + (0.6, 0)$) {};
    \node[vertexA, fill=blue] (aa2l) at ($(v1) + (0, -1)$) {};
    \node[vertexA, fill=blue] (aa2r) at ($(v1) + (0.6, -1)$) {};
    \node[vertex, fill=green] (aa3l) at ($(v1) + (0, -2)$) {};
    \node[vertex, fill=green] (aa3r) at ($(v1) + (0.6, -2)$) {};
    \draw[->] (aa1l)--(aa2l);
    \draw[->,brown] (aa1r)--(aa2r);
    \draw[->] (aa2l)--(aa3l);
    \draw[->,brown] (aa2r)--(aa3r);

    \node[coordinate] (v2) at ($(row1) + (3, 0)$) {};
    \node[vertexCA, fill=red] (ba1) at (v2) {};
    \node[vertexA, fill=blue] (ba2l) at ($(v2) + (-0.5, -1)$) {};
    \node[vertexA, fill=blue] (ba2r) at ($(v2) + (0.5, -1)$) {};
    \node[vertex, fill=green] (ba3l) at ($(v2) + (-0.5, -2)$) {};
    \node[vertex, fill=green] (ba3r) at ($(v2) + (0.5, -2)$) {};
    \draw[->] (ba1)--(ba2l);
    \draw[->,brown] (ba1)--(ba2r);
    \draw[->] (ba2l)--(ba3l);
    \draw[->,brown] (ba2r)--(ba3r);

    \node[coordinate] (v3) at ($(row1) + (6, 1)$) {};
    \node[vertexCA, fill=red] (bba1) at (v3){};
    \node[vertexCA, fill=red] (bba2) at ($(v3) + (0, -1)$) {};
    \node[vertex, fill=green] (bba3l) at ($(v3) + (-0.5, -2)$) {};
    \node[vertexA, fill=blue] (bba3r) at ($(v3) + (0.5, -2)$) {};
    \node[vertex, fill=green] (bba4r) at ($(v3) + (0.5, -3)$) {};
    \draw[->] (bba1)--(bba2);
    \draw[->,brown] (bba2)--(bba3r);
    \draw[->] (bba2)--(bba3l);
    \draw[->,brown] (bba3r)--(bba4r);

    \node[coordinate] (v4) at ($(row1) + (9, 1)$) {};
    \node[vertexCA, fill=red] (bbb1) at (v4) {};
    \node[vertexCA, fill=red] (bbb2) at ($(v4) + (0, -1)$) {};
    \node[vertex, fill=green] (bbb3l) at ($(v4) + (0.5, -2)$) {};
    \node[vertexA, fill=blue] (bbb3r) at ($(v4) + (-0.5, -2)$) {};
    \node[vertex, fill=green] (bbb4r) at ($(v4) + (-0.5, -3)$) {};
    \draw[->,brown] (bbb1)--(bbb2);
    \draw[->] (bbb2)--(bbb3r);
    \draw[->,brown] (bbb2)--(bbb3l);
    \draw[->] (bbb3r)--(bbb4r);

    \node[coordinate] (v5) at ($(row1) + (11.5, 0)$) {};
    \node[vertexCA, fill=red] (bc1l) at (v5){};
    \node[vertexCA, fill=red] (bc1r) at ($(v5) + (1, 0)$) {};
    \node[vertexCA, fill=red] (bc2) at ($(v5) + (0.5, -1)$) {};
    \node[vertex, fill=green] (bc3l) at ($(v5) + (0, -2)$) {};
    \node[vertex, fill=green] (bc3r) at ($(v5) + (1, -2)$) {};
    \draw[->] (bc1l)--(bc2);
    \draw[->,brown] (bc1r)--(bc2);
    \draw[->] (bc2)--(bc3l);
    \draw[->,brown] (bc2)--(bc3r);

    \node[coordinate] (v6) at ($(row1) + (15, 2)$) {};
    \node[vertexA, fill=blue] (bda1) at (v6) {};
    \node[vertexA, fill=blue] (bda2) at ($(v6) + (0, -1)$) {};
    \node[vertexA, fill=blue] (bda3) at ($(v6) + (0, -2)$) {};
    \node[vertexA, fill=blue] (bda4) at ($(v6) + (0, -3)$) {};
    \node[vertex, fill=green] (bda5) at ($(v6) + (0, -4)$) {};
    \draw[->] (bda1)--(bda2);
    \draw[->] (bda2)--(bda3);
    \draw[->,brown] (bda3)--(bda4);
    \draw[->,brown] (bda4)--(bda5);

    \node[coordinate] (v7) at ($(row1) + (18, 2)$) {};
    \node[vertexA, fill=blue] (bdb1) at (v7) {};
    \node[vertexA, fill=blue] (bdb2) at ($(v7) + (0, -1)$) {};
    \node[vertexA, fill=blue] (bdb3) at ($(v7) + (0, -2)$) {};
    \node[vertexA, fill=blue] (bdb4) at ($(v7) + (0, -3)$) {};
    \node[vertex, fill=green] (bdb5) at ($(v7) + (0, -4)$) {};
    \draw[->,brown] (bdb1)--(bdb2);
    \draw[->,brown] (bdb2)--(bdb3);
    \draw[->] (bdb3)--(bdb4);
    \draw[->] (bdb4)--(bdb5);

    \node[coordinate] (v8) at ($(row2)+(-0.5,0)$) {};
    \node[vertexA, fill=blue] (be1l) at (v8){};
    \node[vertexA, fill=blue] (be2l) at ($(v8) + (0, -1)$) {};
    \node[vertexA, fill=blue] (be2r) at ($(v8) + (1, -1)$) {};
    \node[vertexA, fill=blue] (be3) at ($(v8) + (0.5, -2)$) {};
    \node[vertex, fill=green] (be4) at ($(v8) + (0.5, -3)$) {};
    \draw[->] (be1l)--(be2l);
    \draw[->] (be2l)--(be3);
    \draw[->,brown] (be2r)--(be3);
    \draw[->,brown] (be3)--(be4);

    \node[coordinate] (v9) at ($(row2) + (3.5, 0)$) {};
    \node[vertexA, fill=blue] (b2e1l) at (v9){};
    \node[vertexA, fill=blue] (b2e2l) at ($(v9) + (0, -1)$) {};
    \node[vertexA, fill=blue] (b2e2r) at ($(v9) + (-1, -1)$) {};
    \node[vertexA, fill=blue] (b2e3) at ($(v9) + (-0.5, -2)$) {};
    \node[vertex, fill=green] (b2e4) at ($(v9) + (-0.5, -3)$) {};
    \draw[->,brown] (b2e1l)--(b2e2l);
    \draw[->,brown] (b2e2l)--(b2e3);
    \draw[->] (b2e2r)--(b2e3);
    \draw[->] (b2e3)--(b2e4);

    \node[coordinate] (v10) at ($(row2) + (5.5, -1)$) {};
    \node[vertexA, fill=blue] (bf1l) at (v10){};
    \node[vertexA, fill=blue] (bf1r) at ($(v10) + (1, 0)$) {};
    \node[vertexA, fill=blue] (bf2l) at ($(v10) + (0, -1)$) {};
    \node[vertexA, fill=blue] (bf2r) at ($(v10) + (1, -1)$) {};
    \node[vertex, fill=green] (bf3) at ($(v10) + (0.5, -2)$) {};
    \draw[->] (bf1l)--(bf2l);
    \draw[->,brown] (bf1r)--(bf2r);
    \draw[->] (bf2l)--(bf3);
    \draw[->,brown] (bf2r)--(bf3);

    \node[coordinate] (v11) at ($(row2) + (9, -1)$) {};
    \node[vertexA, fill=blue] (ca1) at (v11){};
    \node[vertexA, fill=blue] (ca2l) at ($(v11) + (-0.5, -1)$) {};
    \node[vertexA, fill=blue] (ca2r) at ($(v11) + (0.5, -1)$) {};
    \node[vertex, fill=green] (ca3) at ($(v11) + (0, -2)$) {};
    \draw[->] (ca1)--(ca2l);
    \draw[->,brown] (ca1)--(ca2r);
    \draw[->] (ca2l)--(ca3);
    \draw[->,brown] (ca2r)--(ca3);

    \node[coordinate] (v12) at ($(row2) + (12, 0)$) {};
    \node[vertexA, fill=blue] (cb1) at (v12){};
    \node[vertexA, fill=blue] (cb2) at ($(v12) + (0, -1)$) {};
    \node[vertexA, fill=blue] (cb3) at ($(v12) + (0, -2)$) {}; 
    \node[vertex, fill=green] (cb4) at ($(v12) + (0, -3)$) {};
    \draw[->,brown] (cb1) to [out=330,in=30] (cb3);
    \draw[->] (cb2)--(cb3);
    \draw[->,brown] (cb3)--(cb4);
    \draw[->] (cb1)--(cb2);

    \node[coordinate] (v13) at ($(row2) + (15, 0)$) {};
    \node[vertexA, fill=blue] (c2b1) at (v13){};
    \node[vertexA, fill=blue] (c2b2) at ($(v13) + (0, -1)$) {};
    \node[vertexA, fill=blue] (c2b3) at ($(v13) + (0, -2)$) {};
    \node[vertex, fill=green] (c2b4) at ($(v13) + (0, -3)$) {};
    \draw[->] (c2b1) to [out=210,in=150] (c2b3);
    \draw[->,brown] (c2b2)--(c2b3);
    \draw[->] (c2b3)--(c2b4);
    \draw[->,brown] (c2b1)--(c2b2);

    \node[coordinate] (v14) at ($(row2) + (18, -1)$) {};
    \node[vertexCA, fill=red] (cc1) at (v14){};
    \node[vertexCA, fill=red] (cc2) at ($(v14) + (0, -1)$) {};
    \node[vertex, fill=green] (cc3l) at ($(v14) + (-0.5, -2)$) {};
    \node[vertex, fill=green] (cc3r) at ($(v14) + (0.5, -2)$) {};
    \draw[->] (cc1) to [out=240,in=120] (cc2);
    \draw[->,brown] (cc1) to [out=300,in=60] (cc2);
    \draw[->,brown] (cc2)--(cc3r);
    \draw[->] (cc2)--(cc3l);

    \node[coordinate] (v15) at ($(row3) + (-0.5, -1)$) {};
    \node[vertexA, fill=blue] (cd1l) at (v15){};
    \node[vertexA, fill=blue] (cd1r) at ($(v15) + (1, 0)$) {};
    \node[vertexA, fill=blue] (cd2) at ($(v15) + (0.5, -1)$) {};
    \node[vertex, fill=green] (cd3) at ($(v15) + (0.5, -2)$) {};
    \draw[->] (cd1l)--(cd2);
    \draw[->,brown] (cd1r)--(cd2);
    \draw[->] (cd2) to [out=240,in=120] (cd3);
    \draw[->,brown] (cd2) to [out=300,in=60] (cd3);

    \node[coordinate] (v16) at ($(row3) + (3, 0)$) {};
    \node[vertexA, fill=blue] (ce1) at (v16){};
    \node[vertexA, fill=blue] (ce2) at ($(v16) + (0, -1)$) {};
    \node[vertexA, fill=blue] (ce3) at ($(v16) + (0, -2)$) {};
    \node[vertex, fill=green] (ce4) at ($(v16) + (0, -3)$) {};
    \draw[->] (ce1)--(ce2);
    \draw[->] (ce2) to [out=210,in=150] (ce4);
    \draw[->,brown] (ce2)--(ce3);
    \draw[->,brown] (ce3)--(ce4);

    \node[coordinate] (v17) at ($(row3) + (6, 0)$) {};
    \node[vertexA, fill=blue] (ce1) at (v17){};
    \node[vertexA, fill=blue] (ce2) at ($(v17) + (0, -1)$) {};
    \node[vertexA, fill=blue] (ce3) at ($(v17) + (0, -2)$) {};
    \node[vertex, fill=green] (ce4) at ($(v17) + (0, -3)$) {};
    \draw[->,brown] (ce1)--(ce2);
    \draw[->,brown] (ce2) to [out=330,in=30] (ce4);
    \draw[->] (ce2)--(ce3);
    \draw[->] (ce3)--(ce4);

    \node[coordinate] (v18) at ($(row3) + (9, 0)$) {};
    \node[vertexA, fill=blue] (cf1) at (v18){};
    \node[vertexA, fill=blue] (cf2) at ($(v18) + (0, -1)$) {};
    \node[vertexA, fill=blue] (cf3) at ($(v18) + (0, -2)$) {};
    \node[vertex, fill=green] (cf4) at ($(v18) + (0, -3)$) {};
    \draw[->] (cf1)--(cf2);
    \draw[->] (cf2) to [out=240,in=120] (cf3);
    \draw[->,brown] (cf2) to [out=300,in=60] (cf3);
    \draw[->,brown] (cf3)--(cf4);

    \node[coordinate] (v19) at ($(row3) + (12, 0)$) {};
    \node[vertexA, fill=blue] (c2f1) at (v19){};
    \node[vertexA, fill=blue] (c2f2) at ($(v19) + (0, -1)$) {};
    \node[vertexA, fill=blue] (c2f3) at ($(v19) + (0, -2)$) {};
    \node[vertex, fill=green] (c2f4) at ($(v19) + (0, -3)$) {};
    \draw[->,brown] (c2f1)--(c2f2);
    \draw[->,brown] (c2f2) to [out=240,in=120] (c2f3);
    \draw[->] (c2f2) to [out=300,in=60] (c2f3);
    \draw[->] (c2f3)--(c2f4);

    \node[coordinate] (v20) at ($(row3) + (15, -1)$) {};
    \node[vertexA, fill=blue] (cg1) at (v20) {};
    \node[vertexA, fill=blue] (cg2) at ($(v20) + (0, -1)$) {};
    \node[vertex, fill=green] (cg3) at ($(v20) + (0, -2)$) {};
    \draw[->] (cg1) to [out=240,in=120] (cg2);
    \draw[->,brown] (cg1) to [out=300,in=60] (cg2);
    \draw[->] (cg2) to [out=240,in=120] (cg3);
    \draw[->,brown] (cg2) to [out=300,in=60] (cg3);

  \end{tikzpicture}

  \caption{The set $\GG'_{3,2}$. Labels of the first path $V_1=(v_1^1,v_1^2,v_1^3)$ indicated by black arrows between the nodes and respectively brown arrows for labels of the second path $V_2=(v_2^1,v_2^2,v_2^3)$. Colours and shapes of nodes indicate sink (green $\color{green}\bullet$), `ancestor' (blue~{\small\color{blue} $\blacklozenge$}) and `common-ancestor' (red {\small \color{red} $\blacksquare$}) nodes respectively. These labelled directed acyclic graphs appear in variance calculations of $R(\alpha)$ for $|\alpha|=3$.}\label{fig.GG32}
\end{figure}
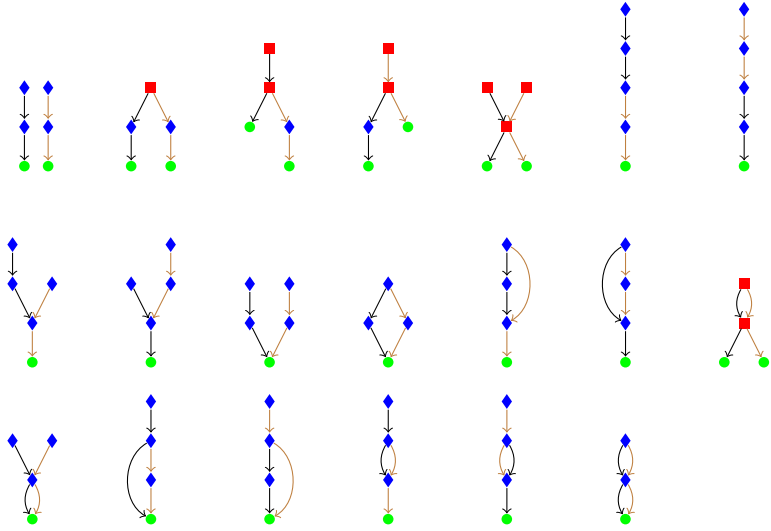

\needspace{8\baselineskip}
\begin{prop}\label{prop.onlystars} Fix $k,r$ and let $\vG$ be a connected digraph in the set $\GG_{k,r}$. If $T_n$ is explosive and $\vG \neq S_{k,r}$ then
  \[
  [\vG]_{T_n}=o\Big([\vS_{k,r}]_{T_n}\Big).
  \]
\end{prop}

\needspace{4\baselineskip}
\begin{proof}
  First observe that $\vS_{k,r}$ has $r$ sink vertices, $(k-2)r$ ancestor vertices and exactly one common-ancestor vertex. Thus by the explosive property of $T_n$

  \[
  [\vS_{k,r}]_{T_n}=\Omega(n^{r}(\ln n)^{(k-2)r}).
  \]
  \needspace{3\baselineskip}
  
  Now fix $\vG\in\GG_{k,r}\backslash S_{k,r}$ and fix a labelling $V^1,\ldots, V^r$ on $\vG$. Again by the explosive property
  \begin{equation}\label{eq.counts}[\vG]_{T_n} = o(n^{|A_0(\vG)|}(\ln n)^{|A_1(\vG)|+1}).\end{equation}
  Hence if $|A_0(\vG)|\leq r-1$ then $[\vG]_{T_n}=o([\vS_{k,r}])$ and so we would be done. Thus we may assume that $A_0(\vG)=r$ and it will suffice to show that $A_1(\vG)<(k-2)r$. 

  As the digraph $\vG$ is connected, each path $V^i$ must have at least one fused vertex. Consider the path labelled $V^i=(v^i_1, \ldots, v^i_k)$. We know $v_k^i$ is a sink vertex and not fused with any other vertex otherwise we would have $A_0(\vG)<r$. If vertex $v^i_j$ on path $V^i$ is fused with another vertex, it must be a vertex on a different path to avoid creating a directed cycle, and so $v^i_j$ and $v^i_{j-1}, \ldots, v^i_{1}$ would become common-ancestors. Thus if $v^i_j$ is fused to another vertex there are at most $(k-j-1)$ ancestor vertices in path $V_i$. Hence $A_1(\vG)\leq (k-2)r$ with equality only if we fused just the source vertices $v_1^i$ of each path $V^i$. But fusing just the source vertices would yield $S_{k,r}$ and so for our digraph $A_1(H)\leq (k-2)r-1$ and we are done.
\end{proof}

We will also need the following lemma in the proof of Proposition \ref{prop.cumulant}. Recall the tree parameter $\U_r^k(T_n)$, defined in Equation \eqref{Uk}, extends the notion of total path length of a tree.  

\begin{lem}\label{lem.upsilon}
Fix $k,r$. If $T_n$ is explosive then
\[
  [\vS_{k,r}]_{T_n}=\U_r^k(T_n)(1+o(1)).
  \]
\end{lem}

\begin{proof}
The star $S_{k, r}$ consists of $r$ directed paths of length $k$ ({\em rays}) with their source vertices fused to a common vertex. Let $\rho$ denote the common vertex, and label all other vertices $v_{i, j}$ for $1\leq i\leq r$ and $2\leq j \leq k$, where $(\rho, v_{i, 2}, v_{i, 3},\dots,v_{i, k})$ makes up ray $i$.

As a warmup we count the number of ways to embed $S_{k, r}$ into a tree $T_n$. Suppose the leaves $v_{1,k},v_{2,k},\dots,v_{r, k}$ are mapped to $u_1,\dots,u_r$ in $T_n$. Then $\rho$ must be mapped to one of the $c(u_1,\dots,u_r)$ common ancestors of $u_1,\dots,u_r$. Having done this, for each $i$ we choose $k-2$ vertices between $u_i$ and $\iota(\rho)$, to which we map $v_{i,2},\dots,v_{i,k-1}$. So the total number of ways is 
  \begin{equation}\label{eq.approxupsilon}
 [\vS_{k,r}]_{T_n} = \sum_{u_1,\dots,u_r} \sum_{\ell=0}^{c(u_1,\dots,u_r)} \prod_{i=1}^r \binom{d(u_i) - \ell}{k-2}.
\end{equation}
We now show that \eqref{eq.approxupsilon} is asymptotically $\U_r^k(T_n)$.  The directed star, $S_{k,r}$ can be constructed by taking $r$ directed paths of length $k$ and fusing their source vertices together to a common vertex. Let $\mathcal{F}_{k,r}$ be the set of graphs obtained by taking $r$ directed paths of length $k$ and fusing one non-sink vertex from each path together to a common vertex and possibly additional pairs of vertices from paths where vertices were at or above this common vertex . So, $S_{k,r} \in \mathcal{F}_{k,r}$, but as for $k>2$ the common fused vertex need not be the source vertex of each path, there may be many other digraphs in $\mathcal{F}_{k,r}$. 

We now count the number of ways to embed $H \in \mathcal{F}_{k,r}$ into a tree $T_n$. Let $\rho$ denote the common vertex to all paths. Label all other unlabelled vertices $v_{i, j}$ for $1\leq i\leq r$ and $1 \leq j \leq k$, where $(v_{i,1}, \rho, v_{i, 3},\dots,v_{i, k})$ makes up ray $i$ if it was the second vertex of path $i$ that was fused.

Recall for any $\vG \in \mathcal{F}_{k,r}$ the sinks of each path are not fused. Suppose the sinks/leaves $v_{1,k},v_{2,k},\dots,v_{r, k}$ are mapped to $u_1,\dots,u_r$ in $T_n$. Then $\rho$ must be mapped to one of the $c(u_1,\dots,u_r)$ common ancestors of $u_1,\dots,u_r$. Having done this, for each $i$ we choose $k-2$ between the root of $T_n$ and $u_i$ to which we map $v_{i,2},\dots,v_{i,k-1}$. (The number of the $k-2$ vertex mapped above and below $\iota(\rho)$ is dependent on which vertex on path $i$ was common vertex in $\vG$). Thus,

$$
\sum_{\vG \in \mathcal{F}_{k,r}} [\vG]_{T_n} = \sum_{u_1,\dots,u_r} c(u_1,\dots,u_r) \prod_{i = 1}^r \binom{d(u_i)}{k-2} = \Upsilon_r^k(T_n).
$$

However there are only finitely many digraphs $\mathcal{F}_{k,r}$ and all of these are connected digraphs also in the set $\mathcal{G}_{k,r}$. Therefore by Proposition \ref{prop.onlystars} 

$$
\sum_{\vG \in \mathcal{F}_{k,r} } [\vG]_{T_n} = [S_{k,r}]_{T_n}(1+o(1))
$$
and we are done.
\end{proof}

\needspace{8\baselineskip}
\section{Labelling stars}

In the proof of Proposition~\ref{prop.cumulant} where we calculate the moments of the distribution of the number of $\alpha$ that occur in a random labelling of our tree we will consider indicators over small subsets of vertices. A star $\vS_{k,\ell}$ can be formed by fusing together $\ell$ length $k$ paths at their source vertices. For $\vS_{k,\ell}$ with a uniform labelling, we calculate the probability each of the $\ell$ paths is labelled with respect to $\alpha$ in Proposition~\ref{prop.probYs}.


 \newcommand{\sss}{\ell}
  
 \m{expression in proposition was: $\frac{1}{\big((k\!-\!1)s+1\big)\big((\alpha_1\!-\!1)!(k\!-\!\alpha_1)!\big)^s}\binom{(k\!-\!1)s}{(\alpha_1\!-\!1)s}^{-1}.$}
 \begin{prop}\label{prop.probYs} Let $\alpha$ be a permutation of length $k$, $\vS_{k,\sss}$ be the digraph defined earlier and let $\lambda: V(\vS_{k,\sss}) \rightarrow [(k\!-\!1)\sss+1]$ be a uniform random labelling of the vertices of $\vS_{k,\sss}$. Then the probability that every $V_i$ induces a labelling of relative order $\alpha$ is,
   \[
   a_{k,\ell}(\alpha)  
   \eqd \frac{\big((\alpha_1-1)\sss\big)!((k-\alpha_1)\sss)!}{\big((\alpha_1-1)!(k-\alpha_1)!\big)^{\sss} \big((k-1)\sss+1\big)!}
   \]
 \end{prop}

\begin{proof}
 First note that for each $V_i$ to induce the relative order $\alpha$, i.e.\ a `correct' labelling there is only one possible label for the root $\rho$. This is obvious if $\alpha_1=1$ since then the root must receive the label~`1'. For general $\alpha_1$, each $V_i\backslash\rho$ must have $\alpha_1-1$ labels less than the label at the root $\lambda(\rho)$ and $k-\alpha_1$ labels greater than $\lambda(\rho)$; hence we must have $\lambda(\rho)=(\alpha_1-1)\sss+1$. Note that we may choose a uniform labelling $\lambda$ by first choosing the label at the root $\lambda(\rho)$ and then choosing uniformly from all labellings of $\vS_{k,r}\backslash \rho$ with the remaining labels. Thus, as there is only one possible label for the root, the probability it is labelled correctly is~$((k\!-\!1)\sss+1)^{-1}$.

It now remains to calculate the probability that the non-root vertices are labelled correctly given that $\lambda(\rho)=(\alpha_1-1)\sss+1$. We count the number of correct labellings. Note there are $(\alpha_1-1)\sss$ labels less than the root i.e.\ `small' labels and $(k-\alpha_1)\sss$ labels greater than the root, `big' labels, remaining. Again each $V_i$ must receive $\alpha_1-1$ of the `small' labels and $k-\alpha_1$ of the `big' labels. As the labels of $V_i$ must induce $\alpha$ once we choose which labels appear on $V_i\backslash \rho$ then they can only be placed in one way. Hence the number of correct labellings of $\vS_{k,\sss}\backslash\rho$ (assuming $\lambda(\rho)=(\alpha_1-1)\sss+1$) is
   \[
   \binom{(\alpha_1\!-\!1)\sss}{\alpha_1\!-\!1,\ldots, \alpha_1\!-\!1}\binom{(k\!-\!\alpha_1)\sss}{k\!-\!\alpha_1,\ldots, k\!-\!\alpha_1}.
   \]
   Note the total number of possible labellings of $\vS_{k,\sss}\backslash\rho$ is $((k-1)\sss)!$ and so the probability of correctly labelling $\vS_{k,\sss}$ is
  \[
   \frac{\big((\alpha_1-1)\sss\big)!((k-\alpha_1)\sss)!}{\big((\alpha_1-1)!(k-\alpha_1)!\big)^\sss \big((k-1)\sss+1\big)!} 
   \]
   and the result follows.
 \end{proof}




\needspace{8\baselineskip}
\section{Cumulants moments}

By exploiting only the explosive property of binary and (whp) of split trees we will prove the moments result for both classes at once, using Proposition~\ref{prop.onlystars}. In particular observe that Theorems~\ref{thm.cumulant} and~\ref{thm.cumulantSPLIT} are both implied by taking Proposition~\ref{prop.cumulant} along with the lemmas proving complete binary trees are explosive and split trees are whp explosive.

To define the constant $D_{\alpha,r}$ used in Proposition \ref{prop.cumulant} and Theorems \ref{thm.cumulant} and \ref{thm.cumulantSPLIT} we use some basic notation of partitions. We write $P(r)$ to indicate the set of all partitions of $[r]$ and note $\{\{1\}\{2,3,4\}\}$ and $\{\{2\}\{1,3,4\}\}$ form different partitions of $[4]$. Given a partition $\pi = \{s_1,\dots,s_\ell\}$ of $\{1,\dots,r\}$ with set sizes $r_i = |s_i|$  we let $|\pi|=\ell$ denote the number of parts in $\pi$. Noting $a_{|\alpha|,\ell}(\alpha)$ is the constant defined in Proposition \ref{prop.probYs} we may now define $D_{\alpha,r}$ by 
\begin{equation}\label{eq.constantD}
   D_{\alpha,r}  \eqd \sum_{\tau \in P(r)} (-1)^{|\tau|-1} (|\tau|-1)! \prod_{s\in \tau}  a_{|\alpha|,|s|}(\alpha). 
 \end{equation}

\needspace{6\baselineskip}
\begin{prop}
     \label{prop.cumulant}
     Suppose $T_n$ is explosive. Let \(\rr_{r}=\rr_r(R(\alpha,T_n))\) be the \(r\)-th cumulant of \(R(\alpha,T_n)\). Then for $r\geq 2$,     \begin{align*}
       \kk_{r} = D_{\alpha,r}\U_{r}^{|\alpha|}(T_n)+o(\U_{r}^{|\alpha|}(T_n)).
       \label{kkIT}
     \end{align*}
   \end{prop}

\begin{proof}
We fix a permutation $\alpha$ with $|\alpha| = k$ and an explosive tree $T_n$ on $n$ nodes, and consider the random variable
$$
X = R(\alpha, T_n) = \sum_{U} \ind\left[\pi(U)\approx \alpha\right],
$$
where we sum over vertex sets $U\subseteq T_n$ of size $|U| = |\alpha|$ which are ordered under the partial ordering of $T_n$, i.e. $U = \{u_1,\dots,u_k\}$ with $u_1 < \dots < u_k$.

In order to calculate the cumulants of $X$, we use mixed cumulants (see e.g. \cite[Section 6.1]{JLR}). Given a set of random variables $X_1,\dots,X_r$, we denote the mixed cumulant by $\kk(X_1,\dots,X_r)$. For now, we only need the following properties.
\begin{enumerate}
\item If $X_1 = X_2 = \dots = X_r$ then $\kk(X_1,\dots,X_r)$ equals the $r$th cumulant $\kk_r(X_1)$ of $X_1$,
\item $\kk(X_1,\dots,X_r)$ is multilinear in $X_1,\dots,X_r$,
\item $\kk(X_1,\dots,X_r) = 0$ if there exists a partition $[r] = A\cup B$ such that $\{X_i : i\in A\}$ and $\{X_i : i\in B\}$ are independent families.
\end{enumerate}
We then have
\begin{align*}
\kk_r(X) = \kk(X, X, \dots,X) & = \kk\left(\sum_{U_1}\ind[\pi(U_1)\approx \alpha], \dots, \sum_{U_r}\ind[\pi(U_r) \approx \alpha]\right) \\
& = \sum_{U_1,\dots,U_r} \kk(\ind[\pi(U_1)\approx \alpha],\dots, \ind[\pi(U_r)\approx \alpha]).
\end{align*}
Now, suppose $\{U_1,\dots,U_r\}$ is a family such that $[r] = A\cup B$ with $U_A = \cup_{i\in A}U_i$ and $U_B = \cup_{i\in B}U_i$ disjoint. Then $\{\ind[\pi(U_i)\approx \alpha] : i\in A\}$ and $\{\ind[\pi(U_i)\approx \alpha] : i\in B\}$ are independent families. Indeed, conditioning on the label sets $\pi(U_A), \pi(U_B)$, the random variables are determined by the internal order given to labels within $U_A$ and $U_B$, respectively, and this order is independent. Saying that the family $\{U_1,\dots,U_r\}$ is {\em connected} if there is no such partition $A\cup B$, it follows that
$$
\kk_r(X) = \sum_{\substack{U_1,\dots,U_r \\ \text{connected}}} \kk(\ind[\pi(U_1)\approx \alpha], \dots, \ind[\pi(U_r)\approx \alpha]).
$$
Let $\{U_1,\dots,U_r\}$ be a connected family. We can write $U_i = \{u_{i,1},\dots,u_{i,k}\}$ with $u_{i, 1} < \dots < u_{i, k}$ for each $i$. Let $H$ be the graph on vertex set $U = U_1\cup\dots\cup U_r$ with an edge from $u_{i, j}$ to $u_{i, j+1}$ for each $i$ and $j < k$. The graph $H$ is a connected member of $\GG_{k, r}$. As the term $\kk(\ind[\pi(U_1)\approx \alpha], \dots, \ind[\pi(U_r)\approx \alpha])$ only depends on the labels of vertices in $U$, it is a function of $H$ which we denote by $\kk(H)$. Then
$$
\kk_r(X) = \sum_{\substack{H\in \GG_{k, r} \\ \text{connected}}} [H]_{T_n} \kk(H).
$$
By Proposition~\ref{prop.onlystars}, this sum is dominated by the term corresponding to $H = S_{k, r}$. We conclude that
$$
\kk_r(X) = (1+o(1)) [S_{k, r}]_{T_n}\kk(S_{k, r}).
$$
But by Lemma \ref{lem.upsilon} $[S_{k, r}]_{T_n}=\U_r^k(T_n)(1+o(1))$ and so it remains only to show $\kk(S_{k, r})=D_{\alpha,r}$.
The mixed cumulant $\kk(X_1,\dots,X_r)$ may be defined by (see e.g. \cite[Section 6.1]{JLR})
$$
\kk(X_1,\dots,X_r) = \sum_{I_1,\dots,I_q} (-1)^{q-1}(q-1)!\prod_{p=1}^q \E{\prod_{j\in I_p} X_j},
$$
where we sum over all partitions of $\{1,\dots,r\}$ into nonempty sets $\{I_1,\dots,I_q\}, q \geq 1$.

Let $V_1,\dots,V_r$ denote the vertex sets of the $r$ ``rays'' of $S_{k, r}$; each $V_i$ has size $k$ and induces a path of length $k$, $V_1\cup\dots\cup V_r$ covers $S_{k, r}$, and the $V_i$ intersect only at the root of $S_{k, r}$. We have
$$
\kk(S_{k, r}) = \kk(\ind[\pi(V_1)\approx \alpha],\dots,\ind[\pi(V_r)\approx \alpha]),
$$
and need to establish $\E{\prod_{j\in I}\ind[\pi(V_j)\approx \alpha]}$ for any $I\subseteq [r]$. By symmetry, this is determined by the size of $I$, and so for $1 \leq \ell\leq r$,
$$
a_{k, \ell} = \E{\prod_{j = 1}^{\ell} \ind[\pi(V_j)\approx \alpha]}.
$$
is the probability that, under a labeling of $S_{k, \ell}$ chosen uniformly at random, each ray respects the permutation $\alpha$ which we calculated in Proposition \ref{prop.probYs}.
Hence we have
\begin{align*}
\kk(S_{k, r}) & = \sum_{I_1,\dots,I_q} (-1)^{q-1}(q-1)!\prod_{p=1}^q a_{k, |I_p|} \\
& = \sum_{q = 1}^r \sum_{r_1 + \dots + r_q = r} \binom{r}{r_1,\dots,r_q} (-1)^{q-1}(q-1)! \prod_{p = 1}^q a_{k, r_p}.
\end{align*}
This may now be written as
  $$
  \kk(S_{k, r}) = \sum_{\pi} (-1)^{|\pi|-1}(|\pi|-1)!\prod_{p\in \pi} a_{k, |p|},
  $$
summing over partitions of $\pi$ of $[r]$ which is the constant $D_{\alpha,r}$ as required. 
\end{proof}

 \bibliographystyle{abbrv}
  \bibliography{trees}

\begin{thebibliography}{10}

\bibitem{albert2018permutations}
M.~Albert, C.~Holmgren, T.~Johansson, and F.~Skerman.
\newblock Permutations in binary trees and split trees.
\newblock In {\em LIPIcs-Leibniz International Proceedings in Informatics},
  volume 110. Schloss Dagstuhl-Leibniz-Zentrum fuer Informatik, 2018.

\bibitem{anders2019rooted}
K.~Anders and K.~Archer.
\newblock Rooted forests that avoid sets of permutations.
\newblock {\em European Journal of Combinatorics}, 77:1--16, 2019.

\bibitem{cai2019inversions}
X.~S. Cai, C.~Holmgren, S.~Janson, T.~Johansson, and F.~Skerman.
\newblock Inversions in split trees and conditional galton--watson trees.
\newblock {\em Combinatorics, Probability and Computing}, 28(3):335--364, 2019.

\bibitem{chauve2001enumerating}
C.~Chauve, S.~Dulucq, and A.~Rechnitzer.
\newblock Enumerating alternating trees.
\newblock {\em Journal of Combinatorial Theory, Series A}, 94(1):142--151,
  2001.

\bibitem{MR1634354}
L.~Devroye.
\newblock Universal limit laws for depths in random trees.
\newblock {\em SIAM Journal on Computing}, 28(2):409--432, 1998.

\bibitem{Finkel1974}
R.~Finkel and J.~Bentley.
\newblock Quad trees a data structure for retrieval on composite keys.
\newblock {\em Acta informatica}, 4(1):1--9, 1974.

\bibitem{flajolet1998analysis}
P.~Flajolet, P.~Poblete, and A.~Viola.
\newblock On the analysis of linear probing hashing.
\newblock {\em Algorithmica}, 22(4):490--515, 1998.

\bibitem{gessel1995enumeration}
I.~M. Gessel, B.~E. Sagan, and Y.-N. Yeh.
\newblock Enumeration of trees by inversions.
\newblock {\em Journal of Graph Theory}, 19(4):435--459, 1995.

\bibitem{MR0142216}
C.~Hoare.
\newblock Quicksort.
\newblock {\em The Computer Journal}, 5(1):10--16, 1962.

\bibitem{holmgren2012novel}
C.~Holmgren.
\newblock Novel characteristics of split trees by use of renewal theory.
\newblock {\em Electronic Journal of Probability}, 17, 2012.

\bibitem{JLR}
S.~Janson, T.~Luczak, and A.~Rucinski.
\newblock {\em Random graphs}, volume~45.
\newblock John Wiley \& Sons, 2011.

\bibitem{lackner2015runs}
M.-L. Lackner and A.~Panholzer.
\newblock Runs in labelled trees and mappings.
\newblock {\em arXiv preprint arXiv:1507.05484}, 2015.

\bibitem{mallows1968inversion}
C.~Mallows and J.~Riordan.
\newblock The inversion enumerator for labeled trees.
\newblock {\em Bulletin of the American Mathematical Society}, 74(1):92--94,
  1968.

\bibitem{panholzer2012limiting}
A.~Panholzer and G.~Seitz.
\newblock Limiting distributions for the number of inversions in labelled tree
  families.
\newblock {\em Annals of Combinatorics}, 16(4):847--870, 2012.

\bibitem{MR0216622}
R.~Pyke.
\newblock Spacings.
\newblock {\em Journal of the Royal Statistical Society. Series B
  (Methodological)}, pages 395--449, 1965.

\bibitem{yan2001generalized}
C.~H. Yan.
\newblock Generalized parking functions, tree inversions, and multicolored
  graphs.
\newblock {\em Advances in Applied Mathematics}, 27(2-3):641--670, 2001.

\end{thebibliography}


\needspace{25\baselineskip}
\section*{Appendix}

\begin{figure}[h]
\centering
$
\begin{array}{|l|l|ccccc|llll}
\hline|\alpha|& \alpha_1\in ? & 1&2&3&4&5 \\
\hline &&&&&&\vspace{-2mm}\\
2 &\{1,2\}& \frac{ 1 }{ 2 }&
\frac{ 1 }{ 2^2 \cdot 3 }&
0&
\frac{ -1 }{ 2^3 \cdot 3 \cdot 5 }&
0\\
&&&&&&\\
3 &\{1,3\}& \frac{ 1 }{ 2 \cdot 3 }&
\frac{ 1 }{ 3^2 \cdot 5 }&
\frac{ 2 }{ 3^3 \cdot 5 \cdot 7 }&
\frac{ - 2 }{ 3^3 \cdot 5^2 \cdot 7 }&
\frac{ - 2^3 }{ 3^4 \cdot 5 \cdot 7 \cdot 11 }\\
&&&&&&\vspace{-3.5mm}\\
3 & \{2\} & \frac{ 1 }{ 2 \cdot 3 }&
\frac{ 1 }{ 2^2 \cdot 3^2 \cdot 5 }&
\frac{ -1 }{ 2^2 \cdot 3^3 \cdot 5 \cdot 7 }&
\frac{ -1 }{ 2^3 \cdot 3^3 \cdot 5^2 \cdot 7 }&
\frac{ 1 }{ 2^2 \cdot 3^4 \cdot 5 \cdot 7 \cdot 11 }\\
&&&&&&\\
4 & \{1,4\} &\frac{ 1 }{ 2^3 \cdot 3 }&
\frac{ 1 }{ 2^6 \cdot 7 }&
\frac{ 1 }{ 2^8 \cdot 5 \cdot 7 }&
\frac{ - 3 }{ 2^{11} \cdot 5 \cdot 7^2 \cdot 13 }&
\frac{ - 3 }{ 2^{12} \cdot 7^2 \cdot 13 }\\
&&&&&&\vspace{-3.5mm}\\
4 & \{2,3\}& \frac{ 1 }{ 2^3 \cdot 3 }&
\frac{ 13 }{ 2^6 \cdot 3^2 \cdot 5 \cdot 7 }&
\frac{ -1 }{ 2^8 \cdot 3^3 \cdot 5 \cdot 7 }&
\frac{ - 5591 }{ 2^{11} \cdot 3^3 \cdot 5^2 \cdot 7^2 \cdot 11 \cdot 13 }&
\frac{ 199 }{ 2^{12} \cdot 3^4 \cdot 5 \cdot 7^2 \cdot 11 \cdot 13 }\\
&&&&&&\\
5&\{1,5\}&
\frac{ 1 }{ 2^3 \cdot 3 \cdot 5 }&
\frac{ 1 }{ 2^2 \cdot 3^4 \cdot 5^2 }&
\frac{ 1 }{ 2^2 \cdot 3^4 \cdot 5^3 \cdot 13 }&
\frac{ 29 }{ 2^3 \cdot 3^7 \cdot 5^4 \cdot 13 \cdot 17 }&
\frac{ - 107 }{ 2^2 \cdot 3^8 \cdot 5^5 \cdot 7 \cdot 13 \cdot 17 }
\\
&&&&&&\vspace{-3.5mm} \\
5&\{2,4\}&
\frac{ 1 }{ 2^3 \cdot 3 \cdot 5 }&
\frac{ 37 }{ 2^6 \cdot 3^4 \cdot 5^2 \cdot 7 }&
\frac{ 53 }{ 2^8 \cdot 3^4 \cdot 5^3 \cdot 7 \cdot 11 \cdot 13 }&
\frac{ - 849839 }{ 2^{11} \cdot 3^7 \cdot 5^4 \cdot 7^2 \cdot 11 \cdot 13 \cdot 17 }&
\frac{ - 1041109 }{ 2^{12} \cdot 3^8 \cdot 5^5 \cdot 7^2 \cdot 11 \cdot 13 \cdot 17 \cdot 19 }
\\
&&&&&&\vspace{-3.5mm} \\
5&\{3\}&
\frac{ 1 }{ 2^3 \cdot 3 \cdot 5 }&
\frac{ 1 }{ 2^6 \cdot 3 \cdot 5^2 \cdot 7 }&
\frac{ - 19 }{ 2^8 \cdot 3^3 \cdot 5^3 \cdot 7 \cdot 11 \cdot 13 }&
\frac{ - 73^2 }{ 2^{11} \cdot 3^3 \cdot 5^4 \cdot 7^2 \cdot 11 \cdot 13 \cdot 17 }&
\frac{ 10061 }{ 2^{12} \cdot 3^4 \cdot 5^5 \cdot 7^2 \cdot 11 \cdot 13 \cdot 17 \cdot 19 }
\\
&&&&&&\\
6&\{1,6\}
&\frac{ 1 }{ 2^4 \cdot 3^2 \cdot 5 }
&\frac{ 1 }{ 2^8 \cdot 3^4 \cdot 11 }
&\frac{ 1 }{ 2^{13} \cdot 3^6 \cdot 11 }
&\frac{ 1 }{ 2^{14} \cdot 3^7 \cdot 7 \cdot 11^2 }
&\frac{ - 19 }{ 2^{19} \cdot 3^9 \cdot 7 \cdot 11^2 \cdot 13 }
\\
&&&&&&\vspace{-3.5mm}\\
6&\{2,5\}
&\frac{ 1 }{ 2^4 \cdot 3^2 \cdot 5 }
&\frac{ 1 }{ 2^8 \cdot 3^2 \cdot 5^2 \cdot 11 }
&\frac{ 509 }{ 2^{13} \cdot 3^6 \cdot 5^3 \cdot 7 \cdot 11 \cdot 13 }
&\frac{ - 233 \cdot 619 }{ 2^{13} \cdot 3^7 \cdot 5^4 \cdot 7 \cdot 11^2 \cdot 13 \cdot 17 \cdot 19 }
&\frac{ - 18928549 }{ 2^{19} \cdot 3^9 \cdot 5^5 \cdot 7 \cdot 11^2 \cdot 13 \cdot 17 \cdot 19 \cdot 23 }
\\
&&&&&&\vspace{-3.5mm}\\
6&\{3,4\}
&\frac{ 1 }{ 2^4 \cdot 3^2 \cdot 5 }
&\frac{ 43 }{ 2^8 \cdot 3^4 \cdot 5^2 \cdot 7 \cdot 11 }
&\frac{ 1 }{ 2^{11} \cdot 3^6 \cdot 5^3 \cdot 7 \cdot 13 }
&\frac{ - 211 \cdot 9341 }{ 2^{15} \cdot 3^7 \cdot 5^4 \cdot 7^2 \cdot 11^2 \cdot 13 \cdot 17 \cdot 19 }
&\frac{ - 47 \cdot 3701 }{ 2^{17} \cdot 3^9 \cdot 5^5 \cdot 7^2 \cdot 11 \cdot 13 \cdot 17 \cdot 19 \cdot 23 }
\\
&&&&&&\\
\hline
\end{array}
$
\caption{A table showing values of $D_{\alpha, r}$ for $\alpha$ of lengths 2 to 6 and moments $r=1,\ldots,5$.}\label{table}
\end{figure}

\end{document}